\documentclass[11pt,reqno]{amsart}

\allowdisplaybreaks[4] 

\makeatletter
\usepackage{amssymb}
\usepackage{amssymb,amsmath,color,comment}
\usepackage{amsbsy}
\usepackage{amsfonts}

\usepackage{enumitem}

\def\marginpar#1{\ignorespaces}

\newtheorem{thm}{Theorem}
\newtheorem{prop}[thm]{Proposition}
\newtheorem{lem}[thm]{Lemma}
\newtheorem{cor}[thm]{Corollary}
\newtheorem{hyp}{Assumption}

\theoremstyle{definition}

\newtheorem{rmq}[thm]{Remark}

\def\eqlabel#1{\def\@currentlabel{#1}}

\def\formula#1{\def\@tempa{#1}\let\@tempb\theequation\def\theequation{%
\hbox{#1}}\def\@currentlabel{(\theequation)}$$}
\def\endformula{\leqno\hbox{(\@tempa)}$$\@ignoretrue\let\theequation\@tempb}

\def\given{\hskip5\p@\relax\vrule\@width.4\p@\hskip5\p@\relax}

\newcommand{\open}[1]{%
\par\normalfont\topsep6\p@\@plus6\p@\trivlist\item[\hskip\labelsep\itshape#1%
\@addpunct{.}]\ignorespaces}

\DeclareRobustCommand{\close}[1]{%
  \ifmmode 
  \else \leavevmode\unskip\penalty9999 \hbox{}\nobreak\hfill
  \fi
  \quad\hbox{$#1$}}

\newlength{\toskip}

\makeatother

\renewcommand{\theequation}{\thesection.\arabic{equation}}

\numberwithin{equation}{section}

\renewcommand\H{\mathcal{H}}

\newcommand{\RR}{{\mathbb R}}
\newcommand{\dd}{\text{d}}

\newcommand{\Div}{\mathop{\rm Div\,}\nolimits}

\newcommand{\Rom}[1]{\uppercase\expandafter{\romannumeral #1\relax}}
\newcommand{\rom}[1]{\lowercase\expandafter{\romannumeral #1\relax}}

\newtheorem*{assumption*}{\assumptionnumber}
\providecommand{\assumptionnumber}{}


\begin{document}

\title[Hypocoercivity and Hypoellipticity in $H^k$ spaces]{Hypocoercivity and global hypoellipticity for the kinetic Fokker-Planck equation in $H^k$ spaces}

\author[Chaoen Zhang]{\textbf{\quad {Chaoen} Zhang \, \, }}
\address{{\bf {Chaoen} ZHANG}\\  Institute for Advanced Study in Mathematics, Harbin Institute of Technology,
NO.92 West Da Zhi St. Harbin, China.} \email{chaoenzhang@hit.edu.cn}

\maketitle

\begin{abstract}
	The purpose of this paper is to extend the hypocoercivity results for the kinetic Fokker-Planck equation in $H^1$ space in Villani's memoir \cite{Villani} to higher order Sobolev spaces. As in the $L^2$ and $H^1$ setting, there is lack of coercivity in $H^k$ for the associated operator. To remedy this issue, we shall modify the usual $H^k$ norm with certain well-chosen mixed terms and with suitable coefficients which are constructed by induction on $k$. In parallel, a similar strategy but with coefficients depending on time (c.f. \cite{Herau}), usually referred as H\'erau's method, can be employed to prove global hypoellipticity in $H^k$. The exponents in our regularity estimates are optimal in short time. Moreover, as in our recent work \cite{GLWZ}, the general results here can be applied in the mean-field setting to get estimates independent of the dimension; in particular, an application to the Curie-Weiss model is presented.
\end{abstract}

\bigskip

\tableofcontents

\section{Introduction}
\label{SectIntr}
We are concerned in this paper with the hypocoercivity and global hypoellipticity of the kinetic Fokker-Planck equation which takes the form
\begin{equation}\label{Eq-kFP}
\frac{\partial f}{\partial t} + v\cdot\nabla_x f - \nabla_xV(x)\cdot\nabla_vf= \Delta_v f + \nabla_v\cdot(vf), \quad t\geq 0
\end{equation}
subject to the initial condition $f(0,x,v)=f_0(x,v)$, where the unknown function $f:=f_t(x,v):=f(t,x,v)$ stands for the density function at time $t$ with position $x\in \RR^d$ and velocity $v\in \RR^d$, and the function $V=V(x):\RR^d\to \RR$ is a smooth potential. We shall always assume that $\int e^{-V(x)}\dd x<\infty$ and thus the kinetic Fokker-Planck equation admits a unique invariant probability measure
\begin{equation}\label{mu}
\dd\mu(x,v)= \frac1{Z_0} e^{-V(x)-\frac{|v|^2}{2}}\dd x\dd v
\end{equation}
where $Z_0= \iint e^{-V(x)-\frac{|v|^2}{2}}\dd x\dd v$ is the normalizing constant.
\medskip

The evolution \eqref{Eq-kFP} preserves mass and positivity. Assume that the initial datum $f_0$ is a probability density function, then, by Ito's formula, $f_t(x,v)$ is the law of a Langevin diffusion process $(X_t, Y_t)_{t\geq 0}$ on $\RR^d\times\RR^d$ evolving according to the stochastic differential equation
\begin{equation}\label{Eq-kFPsde}
\left\{\begin{array}{l}
dX_t=Y_tdt\\
dY_t=-Y_tdt-\nabla V(X_t)dt+\sqrt{2}dB_t
\end{array}\right.
\end{equation}
with the initial law of $(X_0,Y_0)$ being $f_0(x,v)$, where $(B_t)_{t\geq 0}$ is a standard Brownian motion on $\RR^d$.
\medskip

Motivated by the regularizing effect of Kolmogorov's fundamental solution, L. H\"{o}rmander established his celebrated hypoellipticity theorem for second order differential equations in his seminal 1967 paper \cite{Hor67}. The general structure he discovered is usually referred to as H\"{o}rmander's bracket condition (or sometimes simply hypoelliptic structure): let $\mathcal{L}$ be a differential operator in the sum-of-squares form in the sense that
\begin{equation}\label{Hormander}
\mathcal{L}=\sum\nolimits_{i=0}^{n}X_i^2+ X_0 +c
\end{equation}
where  $X_0,X_1,\cdots, X_n$ are smooth vector fields in a domain and $c$ is a smooth function, then H\"{o}rmander's bracket condition of rank $r$ is said to be satisfied if the vector fields generated by iterated Lie brackets
\[
\Big\{X_{j_1},\,\, [X_{j_1},X_{j_2}],\, \cdots\,\,\cdots,\,\,  [X_{j_1},[X_{j_2},\cdots,X_{j_r}]\cdots]\Big\}_{j_i=0,1,\cdots,n}
\]
span the whole tangent space at any point. 
H\"{o}rmander's theorem asserts that the operator $\mathcal{L}$ is hypoelliptic whenever H\"{o}rmander's bracket condition holds. It is worth mentioning that J.J.~Kohn simplified H\"omander's proof by applying the theory of pseudo-differential operators (with the price of non-optimal estimates), and that P.~Malliavin gave a probabilistic proof which is the birth of Malliavin calculus.

Concerning the kinetic Fokker-Planck equation \eqref{Eq-kFP}, by H\"ormander's hypoellipticity theory, it is the interaction between the transport part and the diffusion part, more precisely the basic identity $[v\cdot\nabla_x,\nabla_v]=-\nabla_x$, that results in regularization in the position variable where ellipticity fails though.

It turns out that the long time behaviour is also connected to such an interaction. 
It is well-known that the diffusion part of \eqref{Eq-kFP} is coercive in velocity (in the sense that it admits a spectral gap in velocity) while the transport part is conservative. This lack of coercivity in position variable leads to a large class of local equilibria.
However, as in hypoellipticity theory, it is again the interaction between the diffusion and the transport that results in exponential (as usually expected) convergence to the global equilibrium. Moreover, this seems to be a common feature of many important spatially inhomogeneous kinetic equations.
Such a phenomenon and then the related methods are called ``hypocoercivity", c.f. C.~Villani's memoir \cite{Villani}. The term was suggested to Villani by T.~Gallay to emphasize the links and analogies with hypoellpticity. For instance, at the simplest level, just like the goal of hypoellipticity is to prove regularity estimates in the absence of ellipticity, the goal of hypocoercivity is to prove exponential decay in the absence of coercivity.

We recall the definition in the simplest case: let $L$ be a densely-defined linear operator on a Hilbert space $H$ equipped with a norm $||\cdot||$, then $L$ is called hypocoercive if there exist some constants $C$ and $\lambda>0$ such that
\[
||e^{-tL}h||\leq Ce^{-t\lambda}||h||,
\]
for all $h\in H$ orthogonal to the space spanned by the global equilibrium. Here the constants $C$ and $\lambda$ are supposed to be quantitative or constructive. It is rather straightforward to extend the previous definition to nonlinear PDEs and other kinds of distances, or to other quantities such as entropy.

Despite their close relations, hypocoercivity is indeed distinct from hypoellipticity (for an example we refer to \cite{Herau06}) and can be proved independently. A general strategy for hypocoercivity, systematically developed by Villani in \cite{Villani}, is that one may add some well-chosen auxiliary terms to the usual Lyapunov functional in order to obtain certain coercivity estimates. This strategy might seem simple, but it turns out to be very powerful in the study of the kinetic Fokker-Planck equation and spatially inhomogeneous Boltzmann equation, as shown in \cite{Villani}, as well as of many other kinetic models.

We refer to \cite{HelN} and \cite{Villani} and the references therein for the results and contributions on hypoelliticity and hypocoercivity. Note that basic well-posedness theorems for the kinetic Fokker-Planck equation can also be found there.
\medskip

Let us review some results in the concern of the present article. More or less at the same time, the rate of convergence for solutions of \eqref{Eq-kFP} or \eqref{Eq-kFPsde} has attracted attention of both probability community and PDE community. On the one hand, Wu \cite{Wu01} has established non-quantitative exponential convergence to equilibrium for the ``stochastic Hamiltonian system" (i.e. general kinetic Langevin diffusion) based on Lyapunov functions and Meyn-Treedie's techniques, and Talay \cite{Tal02} proved the exponential convergence in $L^2$ for the kinetic Langevin diffusion. On the other hand, Desvillettes and Villani \cite{DV01} have proved convergence to equilibrium in entropy with quantitative algebraic rates of decay (indeed $O(t^{-\infty})$).
Notably their approach applies to the spatially inhomogeneous Boltzmann equation.

Later, the work of H\'erau and Nier \cite{HN04} gives the first quantitative exponential convergence result (excluding those for the special case of  quadratic potentials) for the evolution equation \eqref{Eq-kFP}, see also some further development in the book \cite{HelN} by Helffer and Nier. Their approach is based on Kohn's method of hypoellipticity theory and spectral analysis.

The research of hypocoercivity for kinetic equations, with an emphasis on constructive rates of convergence, was then put much forward by H\'erau, Mouhot, Villani, and many others. In his memoir \cite[Part \Rom{1}]{Villani}, Villani proved exponential convergence in $H^1$, $L^2$ and entropy under some boundedness (relative or not) assumptions on the Hessian of the potential. One of his key ingredients is to take into consideration higher order terms (namely the Dirichlet energy or Fisher information), and to add auxiliary mixed terms of derivatives which help to make advantages of the basic identity $[v\cdot\nabla_x,\nabla_v]=-\nabla_x$. H\'erau \cite{Herau} also obtained exponential decay in $L^2$ under the bounded Hessian assumption. His approach is based on the spectral methods in the earlier work \cite{HN04}. Another approach to hypocoercivity in $L^2$ under similar assumptions was proposed by Dolbeault, Mouhot and Schmeiser in \cite{DMS15} (for previous works, see \cite{MN} and \cite{DMS09}). Based on a micro-macro decomposition, they constructed an auxiliary operator which helps to make good use of the microscopic and macroscopic coercivity.

Meanwhile, global hypoellipticity for the equation \eqref{Eq-kFP} was also initiated by the quoted works of H\'erau and Villani. While Villani's methods (see \cite[A.21]{Villani}) are based on interpolation inequalities and a system of differential inequalities, H\'erau \cite{Herau} devised a very nice functional to study the short time behaviour of the derivatives. His idea was to distinguish the orders of time for the derivatives in velocity and the ones in position. Moreover the regularity estimates of H\'erau and Villani are in terms of global quantities (i.e. integrals involving derivatives in the whole phase space), contrary to earlier works on hypoellipticity. We shall adapt H\'erau's method for global hypoellipticity in higher order Sobolev spaces. Besides, we refer to the work \cite{GW} by Guillin and Wang in which they obtained some local version 
of H\'erau-Villani's global hypoellipticity estimates by coupling methods; see also \cite{Zx} for a related work concerning Bismut formula. We further remark that uniform-in-time hypoellipticity estimates are of interest for the study of kinetic equations. We also call attentions to recent developments on the regularity estimates for general kinetic Fokker-Planck equations, see for instance Mouhot's ICM report \cite{M18ICM} and the references therein.

In addition, we refer to our recent work \cite{CGMZ} for a relaxation of the bounded Hessian condition with the help of weighted logarithmic Sobolev inequalities. We also refer to \cite{Baud}  and \cite{BauGH} by Baudoin and his collaborators for works by local $\Gamma$-calculus. We also mention \cite{Cal} for another reformulation of Villani's argument by the language of Riemannian geometry, and \cite{GrS} for an extension of Dolbeault-Mouhot-Schmeiser's method together with an application to the so-called spherical velocity Langevin operator. 

Moreover, hypocoercivity can be measured in many different ways, such as entropy, $\varphi$-entropies, total variation, Wasserstein distance, $L^2$ norms, $H^k$ norms, or more general weighted Sobolev norms, see for instance \cite{BCG08}, \cite{BGM10}, \cite{AAS}, \cite{MM16}, \cite{CGMZ}, \cite{DL18}, \cite{Baud}, \cite{EGZ17}(see also related work \cite{Eberle} and \cite{EGZ16}), \cite{E} and the references therein. 
We also call attentions to \cite{HW}, \cite{DEH}, \cite{BDMMS}, \cite{BDS}, \cite{BDLS}, \cite{Cao}, \cite{Cao19} and \cite{GLWZ} for some recent results concerning hypocoercivity with various kinds of potentials. In particular, we highlight \cite{BauGH} and \cite{HM} for some results concerning singular potentials.

\medskip
The purpose of the present article is twofold. ($\rom{1}$) Concerning the long time behaviour, we extend Villani's hypocoercivity theorem in $H^1(\mu)$ to the setting of $H^k(\mu)$, where $k\geq 1$ is an arbitrary integer. For this we introduce a mixed term involving the derivatives $ \nabla_x^k $ and $\nabla_x^{k-1}\nabla_v$ and modify the coefficients in the usual $H^k(\mu)$-semi-norm. ($\rom{2}$) Concerning the short time behaviour, we present higher regularity estimates with optimal exponents in the form of H\'erau-Villani's global hypoellipticity estimates. For this we adapt H\'erau's method and hence use time-dependent coefficients. 
\medskip

\textbf{Plan of the paper.}
We introduce our main results and the notations in Section ~2. The Sections ~3 is devoted to the estimates needed for the proof of our main results. Then the proof of our hypocoercivity result (Theorem \ref{thmHigh}) is presented in Section ~4, while the global hypoellipticity result (Theorem \ref{thmReg}) is proved in section ~5. The proof of an application to the Curie-Weiss model can be found in section ~6. The optimality of the exponents in Theorem \ref{thmReg} is shown in the Appendix A. Some technical lemma concerning quadratic forms is proven in the Appendix B.

\section{Main results and notations}
\label{SectMain}
Let us introduce the basic framework we shall work within. It might be convenient to consider the density function with respect to the invariant measure $\mu$, i.e. the function $$h(t,x,v)= Z_0f(t,x,v) e^{V(x)+\frac{|v|^2}{2}}$$
(see \eqref{mu} for $\mu$ and $Z_0$). Then the evolution equation \eqref{Eq-kFP} becomes
\begin{equation}\label{Eq-kFP*}
\partial_t h + Lh=0
\end{equation}
with $L$ (here and henceforth) being
\[
Lh=-\Delta_v h + v\cdot\nabla_vh +v\cdot\nabla_x h - \nabla_xV(x)\cdot\nabla_vh.
\]
The weighted square-integrable space $L^2(\mu)$, with the norm $||h||_{L^2(\mu)}^2= \int h^2\dd\mu$, is standard for the study of kinetic Fokker-Planck equation. One of the advantages is that the operator $-\Delta_v + v\cdot\nabla_v$ is symmetric in $L^2(\mu)$, while $v\cdot\nabla_x - \nabla_xV(x)\cdot\nabla_v$ is anti-symmetric. Actually, if we set $A:=\nabla_v$, $B:=v\cdot\nabla_x - \nabla_xV(x)\cdot\nabla_v$, then $L$ may be rewritten in H\"ormander's form of a sum of squares
\[
L=A^*A+B
\]
where $A^*$ is the dual operator of $A$ in $L^2(\mu)$. Here we remark that $X^*=-X + g$ for any vector field $X$ and some function $g$, as compared to \eqref{Hormander}. In this article we are mainly concerned with higher order weighted Sobolev spaces $H^k(\mu)$ (with $k$ being a positive integer), of which the notations are specified later in this section.
\medskip

Let us briefly recall Villani's hypocoercive method in the setting of kinetic Fokker-Planck equations. As mentioned in the introduction, the dissipative part in \eqref{Eq-kFP*}, namely the diffusion operator $-\Delta_v+ v\cdot\nabla_v$, only acts on velocity variable and is degenerate in the $x$-directions. This leads to the lack of coercivity in the $x$-variable. It may be illustrated by the computation below,
\begin{align*}
	&-\frac12\frac{\dd }{\dd t} \int h^2\dd \mu = \int |\nabla_v h|^2 \dd\mu,\\
	&-\frac{\dd }{\dd t} \int h\log h\dd \mu = \int \frac{|\nabla_v h|^2}{h}\dd\mu
\end{align*}
where $h$ in the latter equality is assumed to be a positive solution. As one can see, the entropy production is only in the $v$-directions and hence the evolution possesses a large class of local equilibria. This is very different from the Fokker-Planck equation where the entropy production is in all the directions and so one can apply entropy-entropy-production inequalities to deduce quantitative convergence to equilibrium (c.f. \cite{BGL14} or \cite{AMTU}).

To remedy this difficulty, Villani devised certain carefully-chosen Lyapunov functionals for the evolution equation \eqref{Eq-kFP*}. More precisely, he constructed
\begin{enumerate}
	\item in the Hilbertian setting ($L^2$ or $H^1$)
	\[
	((h,h))_{H^1}:=\int h^2 \dd\mu + a \int |\nabla_v h|^2\dd\mu + 2b \int \nabla_vh\cdot\nabla_x h\dd\mu + c \int |\nabla_x h|^2\dd\mu;
	\]
	\item in the entropic setting
	\[
	\int h\log h \dd\mu + a \int \frac{|\nabla_v h|^2}{h}\dd\mu + 2b\int \frac{ \nabla_vh\cdot\nabla_x h}{h}\dd\mu + c \int \frac{|\nabla_x h|^2}{h}\dd\mu
	\]
\end{enumerate}
with some appropriate constants $a,b,c$. The introduction of the mixed terms $\int \nabla_vh\cdot\nabla_x h\dd\mu$ and $\int \frac{ \nabla_vh\cdot\nabla_x h}{h}\dd\mu$ turned out very helpful in the analysis of the long time behaviour: thanks to the commutation relation $[v\cdot\nabla_x,\nabla_v]=-\nabla_x$, it can be used to produce entropy dissipation in the missing $x$-direction.

Let us see what happens in the $H^1$ setting. According to the proof of \cite[Theorem 18, Theorem 35]{Villani}, it holds
\begin{align*}
	((h,Lh))_{H^1} =\, & ||\nabla_v h||^2 + a(||\nabla_v^2h||^2 + ||\nabla_vh||^2+\langle \nabla_vh,\nabla_x h\rangle)  \notag\\
	& +b(2 \langle \nabla^2_vh,\nabla^2_{xv}h\rangle + \langle \nabla_vh,\nabla_x h\rangle+||\nabla_x h||^2 - \langle \nabla_vh,  \nabla^2V\cdot\nabla_v h\rangle) \notag\\
	& +c(||\nabla^2_{xv}h||^2- \langle\nabla_xh,\nabla^2V\cdot\nabla_vh\rangle)
\end{align*}
where $||\cdot||$ and $\langle\cdot,\cdot\rangle$ standard for the norm and inner product in $L^2(\mu)$, c.f. subsection 2.3. Then it is clear that the missing Dirichlet energy in the $x$-direction appears in the temporal derivative of the mixed term  $\langle\nabla_vh,\nabla_x h\rangle$. When $\nabla^2V$ is relatively bounded in the sense of $\eqref{IneqV2}$(see below), the constants $a, b, c$ can be chosen such that $((h,Lh))_{H^1} \geq \lambda' ||h||^2_{\dot{H}^1}$ for some $\lambda'>0$.

Now we recall the Poincar\'e inequality which is somewhat standard for the study of long time behaviour.

\begin{hyp}\label{HypPI}
	Denote $Z_1= \int e^{-V(x)}\dd x$. Assume that the measure $\dd\nu(x):=\frac{1}{Z_1}e^{-V(x)}\dd x$ satisfies a Poincar\'e inequality with constant $\kappa$, i.e. it holds
	\[
	\int \left(g-\int g\dd\nu\right)^2\dd\nu \leq \kappa \int |\nabla g|^2\dd\nu
	\]
	for all functions $g=g(x)\in H^1(\nu)$.
\end{hyp}
This assumption implies the equivalence of the $H^1(\nu)$-norm and the $H^1(\nu)$-seminorm $\|\nabla g\|_{L^2(\nu)}$. Note that the gaussian distribution of the velocity also satisfies the Poincar\'e inequality with $\kappa=1$, we then know the equivalence of the $H^1(\mu)$-norm and the $H^1(\mu)$-seminorm, by the tensorization property of Poincar\'e inequalities. We refer to the monograph \cite[Chapter 4]{BGL14} by Bakry, Gentil and Ledoux for more on Poincar\'e inequalities. See also \cite{BBCG} for some criteria and a simple proof of Poincar\'e inequalities for a very general class of probability measures. 

In particular, whenever $ac>b^2$, the $H^1(\mu)$-seminorm is equivalent to the norm defined by $((\cdot,\cdot))_{H^1}$ due to the Poincar\'e inequality.  By the Gronwall lemma, one can conclude exponential decay in $((\cdot,\cdot))_{H^1}$. Finally one arrives at Villani's hypocoercivity theorem in $H^1$ for the kinetic Fokker-Planck equation,

\begin{thm}(\cite[Theorem 35]{Villani})\label{thmVillani}
	Suppose the Assumption \ref{HypPI} holds. Assume furthermore that the  potential $V\in C^\infty(\RR^d)$ satisfies
	\begin{equation}\label{IneqV2}
		\int |\nabla^2V\cdot \nabla_v g|^2\dd\mu \leq M\bigg(\int |\nabla_v g|^2\dd\mu + \int |\nabla_{xv}^2 g|^2\dd\mu\bigg),
	\end{equation}
	for all $g\in H^2(\mu)$. Then there exist explicitly computable constants $C$ and $\lambda>0$ such that
	\begin{equation}\label{EqThmH2}
		||h_t-\int h_0\dd\mu||_{H^1(\mu)}\leq Ce^{-\lambda t} ||h_0-\int h_0\dd\mu||_{H^1(\mu)}, \quad \quad \forall t\geq 0
	\end{equation}
	where $h_t:=h(t,x,v)$ is the solution to the kinetic Fokker-Planck equation with the initial condition $h_0\in H^1(\mu)$. The constants $C$ and $\lambda$ only depends on $\kappa$ and $M$.
\end{thm}

The inequality \eqref{IneqV2} is slightly different from the one in the original statement of Villani's results \cite{Villani} where \eqref{IneqV2} is verified by the assumption that there exists some positive constant $C$ such that
\[
|\nabla^2 V|\leq C(1+ |\nabla V|).
\]
But the formulation \eqref{IneqV2} is of interest for the generality and flexibility. One of the advantages in this formulation is that the constant $M$ can be independent of the number of particles in the mean-field setting, see for instance in our recent work \cite{GLWZ}. Moreover, it is a special case of the following  boundedness assumption which will be used in our main results:
\begin{hyp}\label{HypBHess}
	There exists some constant $M$ such that
	\begin{equation}\label{IneqVk}
		\int |\nabla_x^{l}V\cdot\nabla_vg|^2 \dd\mu\leq M \bigg(\int |\nabla_vg|^2\dd\mu + \int |\nabla^2_{xv}g|^2\dd\mu\bigg)
	\end{equation}
	for $2\leq l\leq k+1$ and any function $g\in H^2(\mu)$, where
	\[
	|\nabla_x^{l}V\cdot\nabla_vg|^2 := \sum\limits_{ |\alpha|=l-1} \Big|\sum\limits_{j=1}^{d} (D^{\alpha}_x \partial_{x_j}V(x)) \partial_{v_j}g\Big|^2.
	\]
	with $\alpha$ a multi-index in the partial derivative  $D^{\alpha}_x$ in the $x$-variable.
\end{hyp}

This assumption holds trivially for instance whenever the partial derivatives of $V$ of order 2 or more are uniformly bounded, which is the assumption imposed in \cite{Villani}, \cite{Herau}, and \cite{HT16}.

Roughly speaking, Assumption \ref{HypBHess} requires relative boundedness of $\nabla^lV$ ($2\leq l\leq k+1$) as operators; and the inequalities  \eqref{IneqVk} can also be viewed as weighted Poincar\'e inequalities. One can develop Lyapunov type conditions to ensure the validity of such inequalities.
\medskip

We also recall the global hypoellipticity estimates of H\'erau-Villani for the kinetic Fokker-Planck equation.
\begin{thm}[H\'erau\cite{Herau}, Villani\cite{Villani}]
	Let $V$ be a smooth potential satisfying Assumption \ref{HypBHess} with $l=2$. Then there exist constant $C>0$ such that
	\[
	\int |\nabla_x h|^2 \dd\mu \leq Ct^{-3} \int h^2_0 \dd\mu,
	\quad \int |\nabla_v h|^2 \dd\mu \leq Ct^{-1} \int h^2_0 \dd\mu, \quad \quad \mbox{for } 0<t\leq 1
	\]
	for any solution $h_t:=h(t,x,v)$ to \eqref{Eq-kFP*} with initial datum $h_0\in L^2(\mu)$.
\end{thm}
To this end, H\'erau devised a clever Lyapunov functional
\[
\int h^2 \dd\mu + at\int |\nabla_v h|^2\dd\mu + 2bt^2 \int \nabla_vh\cdot\nabla_x h\dd\mu + ct^3 \int |\nabla_x h|^2\dd\mu
\]
with suitable constants $a,b,c$. Then it can be shown that this functional is monotonic in time from which one can conclude the results by standard approximation.

The key point here is to notice the different orders of time of the gradients in different directions. In fact, the choice of the coefficients $t, t^2, t^3$ is not innocent at all. They are related to the intrinsic structure of the kinetic Fokker-Planck equation. As one can see in the appendix A, the global hypoelliptic estimates above are optimal in certain sense. We shall follow H\'erau's construction for the global hypoelliptic estimates in $H^k(\mu)$.
\medskip

\subsection{Main results}
Now we turn to present the main results. Our strategy for the hypocoercivity and hypoellipticity results in $H^k$ closely follows the line of Villani's proof and H\'erau's method. To remedy the degeneracy at the $x$-direction, we introduce a mixed term $ \langle\nabla^{l-1}_x\nabla_v h,\nabla^l_x h \rangle$ (see the subsection 2.3 for the notations) in the $H^l(\mu)$-seminorm for each $1\leq l\leq k$. These mixed terms play the same role as the one in Villani's argument: their temporal derivatives will give the missing pattern $||\nabla^l_xh||^2$ in the temporal derivative of the usual $H^l(\mu)$-seminorm, and so they provide great help to get coercive estimates or monotonicity.
\medskip

Concerning the global hypoellipticity estimates, we have
\begin{thm}\label{thmReg}
	Suppose Assumption \ref{HypBHess} holds. Then there exist explicitly computable constants $C$, depending only on $k$ and $M$, such that
	\begin{equation}\label{EqThm}
	||\nabla_x^{l}\nabla_v^{k-l}h_t||_{L^2(\mu)}\leq Ct^{-(\frac{k}{2}+l)} ||h_0||_{L^2(\mu)}, \quad \mbox{ for } 0<t\leq 1,
	\end{equation}
	where $h_t:=h(t,x,v)$ is the solution to the kinetic Fokker-Planck equation \eqref{Eq-kFP*} with initial datum $h_0\in L^2(\mu)$. Here $||\nabla_x^{l}\nabla_v^{k-l}h||_{L^2(\mu)}$ is some Hilbert-Schmidt norm defined by
	\[
	||\nabla_x^{l}\nabla_v^{k-l}h||^2_{L^2(\mu)} := \sum\limits_{|\alpha|= i, |\beta|=j} \int |D_x^{\alpha}D_v^{\beta}h|^2\dd\mu.
	\]
	Moreover, the exponent $\frac{k}{2}+l$ in the estimate \eqref{EqThm} is sharp.
\end{thm}

Following H\'erau's method, the key of the proof is to construct the Lyapunov functional
\begin{align*}
	\mathcal{F}(t,h_t)
	&= ||h_t||^2 + \left(\sigma_{1,0}t||\nabla_vh_t||^2 + \sigma_{1,1}t^{3}||\nabla_xh_t||^2 + 2\sigma_1 t^2\langle \nabla_v h_t, \nabla_x h_t\rangle \right)\\
	&\quad + \cdots + \cdots + \cdots +\\
	&\quad + \bigg(\sigma_{k,0}t^{k}||\nabla^{k}_v h_t||^2 + \sigma_{k,1}t^{k+2}||\nabla^{k-1}_v\nabla_x h_t||^2 + \cdots + \sigma_{k,k}t^{3k}||\nabla^k_x h_t||^2 \\
	&\quad\quad\quad + 2\sigma_{k}t^{3k-1} \langle \nabla_x^{k-1}\nabla_v h_t, \nabla_x^{k}h_t\rangle \bigg)
\end{align*}
with certain coefficients $\sigma_{l,i}, \sigma_{l}$ ($0\leq i\leq l, 1\leq l\leq k$) such that this functional is monotonic decreasing in time. Indeed, we shall prove a stronger version of monotonicity, i.e.
\[
-\frac{\dd}{\dd t} \mathcal{F}(t,h_t)
\geq \frac{\Lambda_{k}}{t}\left( \sum_{1\leq l\leq k}
t^{3l}||\nabla_x^lh_t||^2+\sum\limits_{0\leq i \leq l\leq k} t^{l+1+2i}||\nabla^{l-i+1}_v\nabla^i_x h_t||^2\right)
\]
for some positive constant $\Lambda_k$, c.f. the Proposition \ref{propReg}. As mentioned above, the terms $||\nabla^l_xh||^2$ ($1\leq l\leq k$), which are important in the proof, are due to the introduction of the mixed terms $ \langle\nabla^{l-1}_x\nabla_v h,\nabla^l_x h \rangle$, c.f. Lemma \ref{lemDsptHk}. The implementation of the proof is presented in the Section~\ref{SectElli}, while the optimality of the exponents can be verified by the fundamental solutions in the special case of quadratic potentials, see for instance in the appendix A.
\medskip

Below is our hypocoercivity result in $H^k(\mu)$, where $k\geq 1$ is a integer.
\begin{thm}\label{thmHigh}
	Suppose Assumptions \ref{HypPI} and \ref{HypBHess} hold. Then there exist explicitly computable constants $C$ and $\lambda>0$, depending only on $\kappa$, $k$ and $M$, such that
	\begin{equation}\label{EqThmk}
	||h_t-\int h_0\dd\mu||_{H^{k}(\mu)}\leq Ce^{-\lambda t} ||h_0-\int h_0\dd\mu||_{H^{k}(\mu)}
	\end{equation}
	for $t\geq 0$, where $h_t=h(t,x,v)$ is the solution to the kinetic Fokker-Planck equation \eqref{Eq-kFP*} with the initial datum $h_0\in H^k(\mu)$.
\end{thm}

The proof of this result can be found in Section \ref{SectCoer}. The crucial ingredient is to construct a twisted $H^k(\mu)$ norm, which is equivalent to the usual $H^k(\mu)$ norm, defined by
\begin{align*}
	&\quad ((h,h))_{H^k}\\
	&= ||h||^2 + \left(\omega_{1,0}||\nabla_vh||^2 + \omega_{1,1}||\nabla_xh||^2 + 2\omega_1 \langle \nabla_v h, \nabla_x h\rangle \right)\\
	&\quad + \cdots + \cdots + \cdots +\\
	&\quad + \bigg(\omega_{k,0}||\nabla^{k}_v h_t||^2 + \omega_{k,1}||\nabla^{k-1}_v\nabla_x h||^2 + \cdots + \omega_{k,k}||\nabla^k_x h||^2 + 2\omega_{k}\langle \nabla_x^{k-1}\nabla_v h, \nabla_x^{k}h\rangle \bigg)
\end{align*}
with suitable coefficients $\omega_{l,i}, \omega_{l}$ ($0\leq i\leq l, 1\leq l\leq k$) such that
\[
((h,Lh))_{H^k}\geq \lambda_{k,0} \bigg(\sum\limits_{1\leq l \leq k,0\leq i\leq l}||\nabla_x^i\nabla_v^{l-i}h||^2 + \sum\limits_{0\leq l\leq k}||\nabla_x^l\nabla^{k+1-l}_v h||^2\bigg)
\]
for some constant $\lambda_{k,0}>0$ (c.f. Proposition \ref{PropCoerHk}). Here we are led to add the mixed terms $\langle \nabla_x^{l-1}\nabla_v h_t, \nabla_x^{l}h_t\rangle$ ($1\leq l\leq k$) in the new norm for the missing $||\nabla^l_xh_t||^2$ ($1\leq l\leq k$) in the dissipation of the usual $H^k$-norm. Combined with Assumption \ref{HypPI}, the preceding estimates implies that  the evolution equation \eqref{Eq-kFP*} is coercive under the new norm defined by $((\cdot,\cdot))_{H^k}$. From this the hypocoercivity in $H^k(\mu)$ follows, just as in Villani's proof.

\begin{rmq}
	One can also combine Theorem \ref{thmReg} with Theorem \ref{thmVillani} to prove \eqref{EqThmk} for $t\geq t_0>0$, see the proof of Corollary \ref{corHigh} below. However, in small time, hypocoercivity means that the quantity under consideration is not explosive; in large time, it requires exponential decay in that quantity; and both of the estimates are in a form with constants independent of the initial data. In this sense, our proof in Section \ref{SectCoer} deals with both short and long time estimates in $H^k$ with $H^k$ initial data simultaneously; while the proof for $t\geq t_0>0$ via regularity estimates in Theorem \ref{thmReg} cannot treat the short-time non-explosion in $H^k$ directly, due to the fact that regularization always takes time (i.e. regularity estimates are valid only for later times, and the constants will explode when approaching the starting time).
	
	Moreover, the proofs of Theorem \ref{thmHigh} and Theorem \ref{thmReg} are presented in a more or less unified style. The former is technically easier since the coefficients in the associated Lyaounov functionals do not depend on time. We also remark that, in the context of Theorem \ref{thmHigh}, it might be possible to choose time-dependent coefficients to obtain both hypocoercive (in large time) and global hypoelliptic estimates (in short time) simultaneously, as in our  work \cite{CGMZ}; but we shall not develop this viewpoint for the sake of technical clarity.
\end{rmq}

An immediate consequence of Theorem \ref{thmHigh} and Theorem \ref{thmReg} is,
\begin{cor}\label{corHigh}
	 Suppose the potential $V\in C^\infty(\RR^d)$ satisfies with some constant $M_0$
	\[
	|D_x^{\alpha}V(x)|\leq M_0, \quad \mbox{ for any } x\in \RR^d
	\]
	for any multi-index $\alpha$ such that $2\leq |\alpha|\leq k+1$. Suppose that the measure $e^{-V(x)}\dd x$ satisfies a Poincar\'e inequality. Then there exist explicitly computable constants $C$ and $\lambda>0$ such that
	\begin{equation}\label{EqThmkc}
	||h_t-\int h_0\dd\mu||_{H^{k}(\mu)}\leq Ct^{-3k/2}e^{-\lambda t} ||h_0-\int h_0\dd\mu||_{L^2(\mu)}
	\end{equation}
	where $h_t=h(t,x,v)$ is the solution to the kinetic Fokker-Planck equation \eqref{Eq-kFP*} with the initial datum $h_0\in L^2(\mu)$.
\end{cor}

\begin{proof}[First proof of Corollary \ref{corHigh}]
	By the assumptions, we can apply Theorem \ref{thmHigh} and Theorem \ref{thmReg}. Since $h_{t}-\int h_0\dd\mu$ satisfies the same evolution equation, we may assume $\int h_0\dd\mu =0$. By Theorem \ref{thmReg}, for $0<t\leq 1$, there exists some positive constants $C'$ such that
	\begin{align*}
		||h_{t}||_{H^{k}(\mu)}
		&\leq C't^{-3k/2} ||h_{0}||_{L^{2}(\mu)}.
	\end{align*}
	By Theorem \ref{thmHigh}, for $t\geq 1$, there exists some positive constants $C''$ and $\lambda$ such that
	\begin{align*}
		||h_{t}||_{H^{k}(\mu)}
		&\leq C''e^{-\lambda(t-1)} ||h_{1}||_{H^{k}(\mu)} \leq C'C''e^{-\lambda(t-1)} ||h_{0}||_{L^{2}(\mu)}
	\end{align*}
	It suffices to take $C=\max\{C'e^{\lambda}, C'C''e^{\lambda}\}$ to ensure \eqref{EqThmkc}.
\end{proof}

\begin{proof}[Second proof of Corollary \ref{corHigh}]
	We can also apply the hypocoercivity results in $H^1$ (c.f. Theorem \ref{thmVillani} or Villani \cite{Villani}) or in $L^2$ (c.f. H\'erau \cite{Herau} or Dolbeault-Mouhot-Schmeiser \cite{DMS15}). For instance, as in the first proof, we may assume $\int h_0\dd\mu =0$ and we can apply Theorem \ref{thmReg} for $0<t\leq 1$. For $t\geq 1$, there exists some positive constants $C'$ such that
	\begin{align*}
		||h_{t}||_{H^{k}(\mu)}
		&\leq C' ||h_{t-1}||_{L^{2}(\mu)}.
	\end{align*}
	By applying the hypocoercivity theorem in $L^2$, there exists positive constants $C$ and $\lambda$ such that
	\[
	||h_{t-1}||_{L^{2}(\mu)}\leq C''e^{-\lambda(t-1)}||h_0||_{L^2(\mu)}
	\]
	which then implies the desired result \eqref{EqThmkc} for $t\geq 1$.
\end{proof}

\begin{rmq}[Further questions] It would be interesting to relax the boundedness assumption~\ref{HypBHess}.
\end{rmq}

\subsection{An example in the mean field setting}
In this subsection, we present an example of the mean field model which shows our results could be independent of the number of particles.

The starting point is the following observation: The constants $C$ and $\lambda$ in Theorems \ref{thmHigh} and \ref{thmReg} have no dependence on the dimension as soon as $M$ (in Assumption \ref{HypBHess}) and $\kappa$ (in Assumption \ref{HypPI}) do so. It has been shown in our recent work \cite{GLWZ} that inequalities in the form of \eqref{IneqVk} are useful to get uniform-in-dimension convergence to equilibrium in $H^1$.


Now consider the potential $V$ of mean-field type, i.e.
\begin{equation}\label{CW-V}
	V(x_1,x_2,\cdots,x_N)= \sum_{i} U(x_i) + \frac{1}{2N}\sum_{j\neq i} W(x_i,x_j)
\end{equation}
where $U(x_i)$ stands for the spatial confinement on the particle at position $x_i\in \RR$, $W(x_i,x_j)$ for the interaction between the particle at $x_i$ and the one at $x_j$, and $N$ the total number of particles. Note that the case of $W(x_i,x_j) = W_0(x_i-x_j)$ was treated in  \cite{GLWZ}. Here we shall focus on the following Curie-Weiss model,
\begin{equation}\label{CW-uw}
	U(x_i)= \beta(\frac{x_i^4}{4}-\frac{x_i^2}{2}), \quad W(x_i,x_j)= -\beta K x_ix_j
\end{equation}
where $\beta>0$ is the inverse temperature, and the model is ferromagnetic or antiferromagnetic according to $K > 0$ or $K < 0$. We remark that the method for Assumption \ref{HypBHess} developed in Section \ref{SectApp1} does not make use of uniform log-Sobolev inequalities, unlike the methods in \cite{GLWZ}.

We shall prove in Section \ref{SectApp1} that Assumption 2 holds with
\[
M=2020(\beta^{2/3}+ \beta^2 + K^4\beta^2)
\]
for any $k\geq 1$, and that Assumption \ref{HypPI} holds with constant $\kappa$ independent of the number $N$ of particles under certain conditions which prevent phase transitions. It follows that the Theorems \ref{thmHigh} and \ref{thmReg} apply to this setting with estimates that do not rely on the number of particles. More precisely,

\begin{prop}[Curie-Weiss model]\label{PropCW}
	Let the potential $V$ be defined by \eqref{CW-V} and \eqref{CW-uw}.
	\begin{itemize}
		\item in the antiferromagnetic case ($K < 0$), when the number $N$ of particles is sufficiently large in the sense that
		\[
		N\geq \frac{2 \beta^{3/2} K}{\sqrt{\pi}} e^{-\beta/4},
		\]
		the Assumption $\ref{HypPI}$ holds with $\kappa=\frac12 \frac{\sqrt{\beta}}{\sqrt{\pi}} e^{-\beta/4}$.
		\item in the ferromagnetic case ($K > 0$), suppose the temperature is high enough in the sense that ($\beta$ is sufficiently small)
		\begin{equation}\label{CW_beta}
		\lambda_1:=\frac{\sqrt{\pi}}{\sqrt{\beta}} e^{\beta/4} -\beta K>0.
		\end{equation}
		Then Assumption $\ref{HypPI}$ holds with $\kappa=1/\lambda_1$.
	\end{itemize}
   In both cases, for $L^2$ initial data, it holds that the distribution $h_t$ becomes smooth at positive times,
   \[
   	||\nabla_x^{l}\nabla_v^{k-l}h_t||_{L^2(\mu)}\leq Ct^{-(\frac{k}{2}+l)} ||h_0||_{L^2(\mu)}, \quad \mbox{ for } 0<t\leq 1;
   \]
   for $H^k$ initial data, it holds that the distribution $h_t$ decays exponentially fast to the equilibrium, namely,
   \[
   ||h_t-\int h_0\dd\mu||_{H^{k}(\mu)}\leq Ce^{-\lambda t} ||h_0-\int h_0\dd\mu||_{H^{k}(\mu)},
   \]
   where the constants in the above estimates are independent of $N$(for $N$ large in the first case).

   In particular, in both cases, starting from any $L^2(\mu)$-initial state, the distribution converges to the equilibrium exponentially fast in $H^k(\mu)$ (for any $k\geq 1$) in large time, and the rates of convergence are independent of the number $N$ of particles.
\end{prop}

It is generally believed that the uniform in the number of particles exponential convergence above holds only in the case of no phase transition. That is exactly the meaning of the anti-ferromagnetic condition (at any temperature) or of the condition \eqref{CW_beta} on the inverse temperature in the ferromagnetic case.

\subsection{Notations}
We follow the notations in \cite{Villani}. Let $\H$ be a Hilbert space equipped with the Hilbert norm $||\cdot||$ and the scalar product $\langle\cdot,\cdot\rangle$. Let $A:\H\rightarrow \H^m$ be a linear densely-defined operator with domain $\mathcal{D}(A)$. So the operator $A$ could be written as a $m$-tuple vector of linear operators on $\H$, say
\[
A=(A_i)_{1\leq i\leq m}^\mathsf{T}=(A_1, A_2, \cdots, A_m)^\mathsf{T}, \quad \mbox{ with } A_i:\H \rightarrow \H.
\]
And its adjoint operator $A^*: \H^m\rightarrow \H$ is then given by
\[
A^*=(A_i^*)_{1\leq i\leq m}= (A_1^*, A_2^*, \cdots, A_m^*),
\]
or more explicitly, for a vector $g=(g_1,g_2,\cdots, g_m)^\mathsf{T} \in \H^m$,
\[
A^*g= (A_1^*, A_2^*, \cdots, A_m^*) (g_1,g_2,\cdots, g_m)^\mathsf{T} =\sum\limits_{1\leq i\leq m} A_i^*g_i.
\]
Therefore the linear operator $A^*A$ has the form
\[
A^*A = (A_1^*, A_2^*, \cdots, A_m^*)(A_1, A_2, \cdots, A_m)^\mathsf{T}=\sum\limits_{1\leq i\leq m} A_i^*A_i.
\]
Given two operators $B_1,B_2: \H \rightarrow \H$, their commutator is defined by $[B_1,B_2]:=B_1B_2-B_2B_1$. We stress that the commutator $[A,B]$ of $A:\H \rightarrow \H^m$ and $B: \H \rightarrow \H$, should be understood as the $m$-tuple operator-valued vector $([A_i,B])^\mathsf{T}_{1\leq i\leq m}$; similarly, the commutator of two operator-valued vectors should be understood as an operator-valued matrix in the same way.

Hereafter, we set $\H= L^2(\mu)$. Set two operators $A: \H\rightarrow \H^d $ and $B:\H\rightarrow \H$ as
\[
A:=(A_i)^\mathsf{T}_{1\leq i\leq d}=\nabla_v, \quad\quad B:= v\cdot\nabla_x - \nabla V(x)\cdot \nabla_v,
\]
then
\[
A^*= -\Div_v + v\cdot, \quad\quad B^*=-B,
\]
and
\[
A^*A= -\Delta_v + v\cdot\nabla_v.
\]

We shall mainly consider the weighted $H^k$ space $H^k(\mu)$ (with $k$ being a positive integer) with the norm
\[
||h||_{H^k(\mu)}^2= \sum\limits_{\alpha,\beta: |\alpha|+|\beta|\leq k}\int |D_x^{\alpha}D_v^{\beta}h|^2\dd\mu
\]
where $\alpha, \beta$ are multi-indexes of respective order $|\alpha|$ and $|\beta|$, and the partial derivative $D_x^{\alpha}D_v^{\beta}g$ is given as usual by
\[
D_x^{\alpha}D_v^{\beta}g = \frac{\partial^{|\alpha|+|\beta|} g}{\partial x_1^{\alpha_1}\cdots \partial x_d^{\alpha_d}\partial v_1^{\beta_1}\cdots \partial v_d^{\beta_d}}.
\]
To avoid heavy notations, we shall denote the $L^2(\mu)$ norm by $||\cdot||$, the $L^2(\mu)$ scalar product by $\langle \cdot,\cdot\rangle$ without subscripts referring to the reference measure $\mu$. We should also mention that $\langle\cdot,\cdot\rangle$ might be used for the scalar product in Euclidean spaces in certain occasions such as in local computations or in linear algebra. The notation $\dot{H}^{k}(\mu)$ stands for the homogeneous Sobolev space with the semi-norm defined by
\[
||h||_{\dot{H}^k(\mu)}^2= \sum\limits_{\alpha,\beta: |\alpha|+|\beta|= k}\int |D_x^{\alpha}D_v^{\beta}h|^2\dd\mu.
\]

For any non-negative integers $i,j$ and suitably-differentiable function $g$, we denote
\[
\nabla_x^i\nabla_v^{j}g := (D_x^{\alpha}D_v^{\beta}g)_{\alpha,\beta}
\]
where $\alpha, \beta$ run over multi-indexes of respective order $|\alpha|=i$ and $|\beta|=j$. In particular, the squared Hilbert-Schmidt norm of $\nabla_x^i\nabla_v^{j}g$, denoted by $|\nabla_x^i\nabla_v^{j}g|$, is defined via
\[
|\nabla_x^i\nabla_v^{j}g|^2
:= \sum\limits_{|\alpha|= i, |\beta|=j} |D_x^{\alpha}D_v^{\beta}g|^2,
\]
and so
\begin{equation}\label{NormHS}
||\nabla_x^i\nabla_v^{j}g||_{L^2(\mu)}^2
:=\int |\nabla_x^i\nabla_v^{j}g|^2\dd\mu = \sum\limits_{|\alpha|= i, |\beta|=j} \int |D_x^{\alpha}D_v^{\beta}g|^2\dd\mu.
\end{equation}
Note that the definition of $|\nabla_x^{l}V\cdot\nabla_vh|$ in Assumption \ref{HypBHess} is consistent with the above convention since
\[
\nabla_x^{l}V\cdot\nabla_vh := \sum\nolimits_{j=1}^{d}\big(\nabla_x^{l-1} \partial_{x_j}V\big) \partial_{v_j}h= \left(\sum\nolimits_{j=1}^{d}(D^{\alpha}_x (\partial_{x_j} V)\cdot \partial_{v_j}h\right)_{\alpha: |\alpha|=l-1}
\]
where the summation and multiplication in the middle are componentwise.

The Hilbertian structure, after polarization, induces a scalar product between the matrices of partial derivatives. For instance, consider $\nabla^2_{xv}:=\nabla_{x}\nabla_{v}:= (\partial_{x_i}\partial_{v_j})_{(i,j): 1\leq i,j\leq d}$ and $\nabla^2_{vx}=\nabla_{v}\nabla_{x}$(note that these two matrices are not identical), we have
\[
\langle \nabla_{v}\nabla_{x}g,\nabla_{x}\nabla_{v}h\rangle = \int \sum\limits_{i,j} (\partial^2_{v_ix_j}g) (\partial^2_{x_iv_j}h) \dd\mu.
\]
Here is one more example we shall use later. For nonnegative integers $m_2$ and $ l\leq m_1$,
\begin{align*}
\langle\nabla_x^{m_1-l}\nabla_v\nabla_x^{l-1}\nabla_v^{m_2}g, \nabla_x^{m_1}\nabla_v^{m_2}h\rangle
&=
\sum_{\alpha_1,\alpha_2,\alpha_3,\alpha_4}
\langle D_x^{\alpha_1}D_v^{\alpha_2}D_x^{\alpha_3}D_v^{\alpha_4}g,
D_x^{\alpha_1}D_x^{\alpha_2}D_x^{\alpha_3}D_v^{\alpha_4}h\rangle\\
&= \sum_{\alpha_1,\alpha_2,\alpha_3,\alpha_4} \int(D_x^{\alpha_1}D_v^{\alpha_2}D_x^{\alpha_3}D_v^{\alpha_4}g) (D_x^{\alpha_1}D_x^{\alpha_2}D_x^{\alpha_3}D_v^{\alpha_4}h) \dd\mu
\end{align*}
where $\alpha_1$, $\alpha_2$ $\alpha_3,\alpha_4 $ run over multi-indexes of respective order $m_1-l, 1, l-1, m_2$.

\section{Preliminary estimates in $H^{k}(\mu)$}
\label{SectPre}
The goal of this section is to provide estimates for the temporal derivatives of the terms in the twisted $H^k$ norm. Such estimates will then be applied to obtain hypocoercivity and global hypoellpticity in the next sections.

To begin with, let us introduce the following commutation relations which will play an essential role. Recall that 
\[
A=\nabla_v, \quad B=v\cdot\nabla_x - \nabla_xV(x)\cdot\nabla_v,\quad L=A^*A+B.
\]

\begin{lem}\label{lemComm}  By direct computation, it holds
	\begin{enumerate}[label=(\arabic*)]
		\item $[A,A^*]=I$, i.e. $[A_i,A_j^*]=\delta_{ij}$;
		\item $C:=[A,B]=\nabla_x$;
		\item $R:=[C,B]=[\nabla_x, v\cdot\nabla_x - \nabla V(x)\cdot \nabla_v]= -\nabla^2V(x)\cdot\nabla_v$.
	\end{enumerate}
	Note also that $A$ commutes with both itself and $C$.
\end{lem}

\subsection{Temporal derivatives}
Let $m_1,m_2$ be non-negative integers such that $m_1+m_2=k$. Set
\begin{align}
  T_{m_1,m_2}^A&:= \langle \nabla_x^{m_1}\nabla^{m_2}_{v}A^*Ah, \nabla_x^{m_1}\nabla^{m_2}_{v}h \rangle, \\
  T_{m_1,m_2}^B&:= \langle \nabla_x^{m_1}\nabla^{m_2}_{v}Bh, \nabla_x^{m_1}\nabla^{m_2}_{v}h \rangle, \\
  T_{mix}^A &:= \langle \nabla_x^{k-1}\nabla_v A^*Ah, \nabla_x^{k}h\rangle + \langle \nabla_x^{k-1}\nabla_vh, \nabla_x^{k}A^*Ah\rangle,\\
  T_{mix}^B &:= \langle \nabla_x^{k-1}\nabla_v Bh, \nabla_x^{k}h\rangle + \langle \nabla_x^{k-1}\nabla_vh, \nabla_x^{k}Bh\rangle,
\end{align}
where the pairing is as stated as in section 2, and we also note that $A^*Ah$ and $Bh$ are scalar functions.  Then we have the following result.

\begin{lem}\label{lemDsptHk}Let $h\in \mathcal{S}(\RR^{2d})$ be a rapidly decreasing function. Then
\begin{align}
\label{EqDsptHk1}
  T_{m_1,m_2}^A&= ||\nabla_x^{m_1}\nabla_v^{m_2+1}h||^2 + m_2 ||\nabla_x^{m_1}\nabla_v^{m_2}h||^2,
  \end{align}
  \begin{align}
\label{EqDsptHk2}
  T_{m_1,m_2}^B &=\sum\limits_{l=1}^{m_2}\langle \nabla_x^{m_1}\nabla_v^{m_2-l}\nabla_x\nabla_v^{l-1}h, \nabla_x^{m_1}\nabla_v^{m_2}h\rangle \notag\\
   & \quad +\sum\limits_{l=1}^{m_1} \langle\nabla_x^{m_1-l}(-\nabla^2V\cdot\nabla_v)\nabla_x^{l-1}\nabla_v^{m_2}h, \nabla_x^{m_1}\nabla_v^{m_2}h\rangle,
\end{align}
\begin{align}
\label{EqDsptHk3}
 T_{mix}^A &= 2\langle \nabla_x^{k-1}\nabla_v^2 h, \nabla^k_x\nabla_v h\rangle + \langle \nabla_x^{k-1}\nabla_v h, \nabla^k_xh\rangle,
 \end{align}
 \begin{align}
\label{EqDsptHk4}
 T_{mix}^B &= ||\nabla^k_x h||^2 + \sum\limits_{l=1}^{k-1} \langle \nabla_x^{k-l-1}(-\nabla^2V\cdot\nabla_v)\nabla_x^{l-1}\nabla_v h, \nabla^k_x h\rangle\notag\\
 &\quad + \sum\limits_{l=1}^{k} \langle \nabla^{k-1}_x\nabla_v h, \nabla_x^{k-l}(-\nabla^2V\cdot\nabla_v)\nabla_x^{l-1} h\rangle
\end{align}
where $-\nabla^2V\cdot\nabla_v$ is understood as a twisted gradient in the pairing.
\end{lem}

We may find that $T_{m_1,m_2}^A$ does not contain the term $||\nabla^{k}_xh||^2$ for all $m_1,m_2$ with  a given sum $m_1+m_2=k$, and that the terms in the expression of $T_{m_1,m_2}^B$ might be non-positive. This is the reason to introduce a mixed term $\langle \nabla_x^{k-1}\nabla_v h, \nabla_x^{k}h\rangle$ which helps to obtain $||\nabla^{k}_xh||^2$, as shown in the following equality
\[
\langle\nabla_x^{k-1}\nabla_v Bh, \nabla_x^{k}h \rangle= \langle\nabla_x^{k-1}[\nabla_v, B]h+\nabla_x^{k-1}B\nabla_vh, \nabla_x^{k}h \rangle= ||\nabla^{k}_xh||^2 + \langle\nabla_x^{k-1}B\nabla_vh, \nabla_x^{k}h \rangle
\]
or in the expression in \eqref{EqDsptHk4}.

Before turning to the proof of Lemma \ref{lemDsptHk}, we shall present some detailed computation for $k=2$ first which might be helpful to facilitate the reading. Let us consider the temporal derivative of $\langle\nabla_{xv}^2 h, \nabla^2_x h\rangle + \langle\nabla_{xv}^2 h, \nabla^2_x h\rangle$. The desired equalities corresponding to \eqref{EqDsptHk3} and \eqref{EqDsptHk4} are respectively
\begin{align*}
T_{A}&:=\langle\nabla_{xv}^2 A^*Ah, \nabla^2_x h\rangle + \langle\nabla_{xv}^2 h, \nabla^2_x A^*Ah\rangle
=2\langle \nabla_x\nabla_v^2 h, \nabla_x^2 \nabla_v h\rangle + \langle \nabla^2_{xv} h, \nabla_{x}^2h \rangle,\\
T_{B}&:=\langle\nabla_{xv}^2 Bh, \nabla^2_x h\rangle + \langle\nabla_{xv}^2 h, \nabla^2_x Bh\rangle\\
&=  -\langle \nabla^2V\cdot \nabla^2_v h, \nabla_x^2 h\rangle + ||\nabla_x^2 h||^2 - \langle \nabla^2_{xv} h,\nabla^2V \cdot\nabla^2_{vx} h\rangle - \langle \nabla^2_{xv}h, \nabla_x(\nabla^2V\cdot\nabla_v h)\rangle.
\end{align*}
To this end, we write everything (especially the pairing in the Hilbertian structure) in detailed subscripts. For the term $T_A$, we can use the fact $[A_j,A_k^*]=\delta_{jk}$ to find
	\begin{align*}
		T_{A} &=\sum\limits_{i,j,k}\Big( \langle C_iA_jA_k^*A_k h, C_iC_j h\rangle + \langle C_iA_j h, C_iC_j A_k^*A_k h\rangle \Big)\\
		&=\sum\limits_{i,j,k} \Big(\langle C_i(A_k^*A_j + [A_j,A_k^*])A_k h, C_iC_j h\rangle + \langle C_iA_j h, A_k^*C_iC_jA_k h\rangle\Big) \\
		&= \sum\limits_{i,j,k}\Big(\langle A_k^*C_iA_j A_k h + \delta_{jk}C_iA_k h, C_iC_j h\rangle + \langle A_kC_iA_j h, C_iC_jA_k h \rangle\Big)\\
		&=\sum\limits_{i,j,k} \langle C_iA_j A_k h , A_kC_iC_j h\rangle + \sum\limits_{i,j}\langle C_iA_j h, C_iC_j h\rangle + \sum\limits_{i,j,k}\langle C_iA_j A_k h, C_iC_jA_k h \rangle \\
		&= 2\langle CA^2h, C^2A h\rangle + \langle CA h, C^2 h\rangle.
	\end{align*}
For the term $T_B$, we need to apply the commutation relations between $B$ and $A$, $C$ as listed in Lemma \ref{lemComm}. Recall that $R=[C,B]= -\nabla^2V(x)\cdot\nabla_v$.
	\begin{align*}
		\langle\nabla_{xv}^2 Bh, \nabla^2_x h\rangle
		&= \sum\limits_{i,j}\langle C_iA_jB h, C_iC_j h\rangle = \sum\limits_{i,j}\langle C_i(BA_j + [A_j,B] ) h, C_iC_j h\rangle \\
		&= \sum\limits_{i,j}\langle (BC_iA_j + [C_i,B]A_j + C_iC_j) h, C_iC_j h\rangle \\
		&= \sum\limits_{i,j}\Big(\langle BC_iA_j  h, C_iC_j h\rangle + \langle R_iA_j h, C_iC_j h\rangle + \langle C_iC_j h, C_iC_j h\rangle\Big);\\
		\langle\nabla_{xv}^2 h, \nabla^2_x Bh\rangle
		&=  \sum\limits_{i,j} \langle C_iA_j h, C_iC_j Bh\rangle
		= \sum\limits_{i,j}\langle C_iA_j h, C_i(BC_j + [C_j,B] )h\rangle \\
		&= \sum\limits_{i,j}\langle C_iA_j h, (BC_iC_j + [C_i,B]C_j + C_i R_j)h\rangle\\
		&= \sum\limits_{i,j} \Big(\langle C_iA_j h, BC_iC_j h\rangle + \langle C_iA_j h, R_iC_j h\rangle + \langle C_iA_j h, C_iR_j h\rangle\Big).
	\end{align*}
	Hence we obtain, by the anti-symmetry of $B$,
	\begin{align*}
		T_{B}  
		&=\sum\limits_{i,j}\Big(  \langle R_iA_j h, C_iC_j h\rangle + \langle C_iC_j h, C_iC_j h\rangle 
		+ \langle C_iA_j h, R_iC_j h\rangle + \langle C_iA_j h, C_iC_j h\rangle\Big)\\
		&= \langle RAh, C^2h\rangle + ||C^2h||^2 + \langle CA h, RC h\rangle + \langle CAh, CR h\rangle.
	\end{align*}
As a consequence, the equalities \eqref{EqDsptHk3} and \eqref{EqDsptHk4} hold with $k=2$.

\begin{proof}[Proof of Lemma \ref{lemDsptHk}]
(0). We collect some commutation relations first. Let $l_1,l_2$ be positive integers, then
\begin{align}\label{EqHkCommu1}
   &[A^{l_2},B] 
   =\sum\nolimits^{l_2}_{l=1} A^{l_2-l}CA^{l-1}; \\
 \label{EqHkCommu2}  &[C^{l_1},B]
 = \sum\nolimits^{l_1}_{l=1} C^{l_1-l}RC^{l-1} ; \\
  \label{EqHkCommu3} & [C^{l_1}A^{l_2},B]
  = \sum\nolimits^{l_2}_{l=1} C^{l_1}A^{l_2-l}CA^{l-1} + \sum\nolimits^{l_1}_{l=1} C^{l_1-l}RC^{l-1}A^{l_2}.
\end{align}
They may be deduced by induction from commutation relations between $B$ and $A$ or $C$ . We only provide a proof for \eqref{EqHkCommu2} here, since $\eqref{EqHkCommu1}$ may be proved in the very same manner, and \eqref{EqHkCommu3} follows from the two preceding equalities. For \eqref{EqHkCommu2}, the case of $l_1=1$ is exactly the commutation relation between $B$ and $C$. Now assume that \eqref{EqHkCommu2} holds for $l_1-1$ ($l_1\geq 2$), i.e.
\[
 [C^{l_1-1},B] = \sum\nolimits^{l_1-1}_{l=1} C^{l_1-1-l}RC^{l-1}
\]
then we obtain
\begin{align*}
  [C^{l_1},B] & = C(C^{l_1-1}B) - BC^{l_1}= C [C^{l_1-1},B] + [C,B]C^{l_1-1} \\
   & =C\sum\nolimits^{l_1-1}_{l=1} C^{l_1-1-l}RC^{l-1} + RC^{l_1-1} \mbox{ (by assumption for $l_1-1$ )}\\
   &=\sum\nolimits^{l_1}_{l=1} C^{l_1-l}RC^{l-1}
\end{align*}
as desired. Therefore $\eqref{EqHkCommu2}$ holds for any positive integer $l_1$ by induction.

Another commutation relation might also be useful, namely,
\begin{equation}\label{EqHkCommu4}
  \sum\nolimits_{j=1}^{d} A_{i_1} A_{i_2}\cdots  A_{i_m} A_{j}^* A_{j}= \sum\nolimits_{j=1}^{d}A_{j}^* A_{i_1} A_{i_2}\cdots  A_{i_m} A_{j} +  m A_{i_1} A_{i_2}\cdots  A_{i_m}.
\end{equation}
Omitting the subscripts $i_1,i_2,\cdots,i_m$, it may be written as
\[
 \sum\nolimits_{j=1}^{d} A^m A_{j}^* A_{j}= \sum\nolimits_{j=1}^{d}A_{j}^* A^mA_{j} +  mA^m.
\]
Again it is a simple application of the commutation relation $[A_i,A_j^*]=\delta_{ij}$ and it may be proved by induction on $m$. \vspace{0.314cm}

(1). Now we compute $T_{m_1,m_2}^A$ and $T_{m_1,m_2}^B$. Let us prove \eqref{EqDsptHk1} first,
\begin{align*}
  T_{m_1,m_2}^A
   & = \langle C^{m_1}A^{m_2}A^*A h, C^{m_1}A^{m_2}h\rangle
    = \sum_{j=1}^{d} \langle C^{m_1}A^{m_2}A_j^*A_j h, C^{m_1}A^{m_2}h\rangle\\
   & = \sum_{j=1}^{d} \langle C^{m_1}A_j^*A^{m_2}A_j h, C^{m_1}A^{m_2}h\rangle + m_2 \langle C^{m_1}A^{m_2} h, C^{m_1}A^{m_2}h\rangle \quad\mbox{  by \eqref{EqHkCommu4}}\\
   &= \sum_{j=1}^{d} \langle C^{m_1}A^{m_2}A_j h, A_jC^{m_1}A^{m_2}h\rangle + m_2 || C^{m_1}A^{m_2} h||^2 \\
   & = ||C^{m_1}A^{m_2+1} h||^2 + m_2 || C^{m_1}A^{m_2} h||^2
\end{align*}
where the two last equalities follow from the fact both $A^*$ and $A$ commutate with $C$.

Then we prove \eqref{EqDsptHk2}. As a consequence of \eqref{EqHkCommu3}, we find
\begin{align*}
  T_{m_1,m_2}^B
   & = \langle C^{m_1}A^{m_2}B h, C^{m_1}A^{m_2}h\rangle \\
   & = \sum\limits_{l=1}^{m_2}\langle C^{m_1}A^{m_2-l}CA^{l-1} h, C^{m_1}A^{m_2}h\rangle  \\
   &\quad + \sum\limits_{l=1}^{m_1}\langle C^{m_1-l}RC^{l-1}A^{m_2} h, C^{m_1}A^{m_2}h\rangle + \langle BC^{m_1}A^{m_2} h, C^{m_1}A^{m_2}h\rangle
\end{align*}
Then \eqref{EqDsptHk2} follows since $B$ is anti-symmetric,
\[
\langle BC^{m_1}A^{m_2} h, C^{m_1}A^{m_2}h\rangle=0.
\]

(2). Similarly we proceed with the computations of $T_{mix}^A$ and $T_{mix}^B$. Thanks to \eqref{EqHkCommu4} and $[A,C]=[A^*,C]=0$, we know
\begin{align*}
  T_{mix}^A
  & = \sum\limits_{j=1}^{d} \big(\langle C^{k-1}AA^*_jA_j h, C^kh\rangle + \langle C^{k-1}A h, C^kA^*_jA_j h\rangle \big)\\
  & =\langle C^{k-1}A h, C^kh\rangle + \sum\limits_{j=1}^{d} \langle A^*_jC^{k-1}AA_j h, C^kh\rangle +  \sum\limits_{j=1}^{d} \langle C^{k-1}A h, A^*_jC^kA_j h\rangle\\
  & =\langle C^{k-1}A h, C^kh\rangle + \sum\limits_{j=1}^{d} \langle C^{k-1}AA_j h,C^k A_jh\rangle +  \sum\limits_{j=1}^{d} \langle C^{k-1}A A_j h,C^kA_j h\rangle \\
  &  =\langle C^{k-1}A h, C^kh\rangle + 2\langle C^{k-1}A^2 h,C^k Ah\rangle,
\end{align*}
\begin{align*}
  T_{mix}^B
  &= \langle C^{k-1}AB h, C^kh\rangle + \langle C^{k-1}A h, C^kB h\rangle\\
  &= \langle C^{k}h, C^kh\rangle + \sum\limits_{l=1}^{k-1}\langle C^{k-1-l}RC^{l-1}A h, C^k h\rangle + \langle BC^{k-1}A h, C^kh\rangle\\
  &\quad + \sum\limits_{l=1}^{k}\langle C^{k-1}A h, C^{k-l}RC^{l-1} h\rangle + \langle C^{k-1}A h, BC^kh\rangle \quad\mbox{ by \eqref{EqHkCommu3}}\\
  &= ||C^kh||^2 + \sum\limits_{l=1}^{k-1}\langle C^{k-1-l}RC^{l-1}A h, C^k h\rangle + \sum\limits_{l=1}^{k}\langle C^{k-1}A h, C^{k-l}RC^{l-1} h\rangle
\end{align*}
where the last equality holds since $B$ is anti-symmetric.
\end{proof}

\subsection{The estimates} To abbreviate the notations, let us introduce
\begin{align}
Z&:=\big(\sum\limits_{1\leq l \leq k-1}||h||_{\dot{H}^l}^2 + \sum\limits_{0\leq l\leq k-1}||\nabla_x^l\nabla^{k-l}_v h||^2\big)^{\frac12}\\
W&:=(W_x,W_0,W_1,\cdots, W_{k})^{\mathsf{T}}\in\RR^{k+2}\\
& \mbox{where } W_x:= ||\nabla_x^{k}h||,\quad W_l:=||\nabla_x^l\nabla^{k+1-l}_v h||, \quad 0\leq l\leq k.
\end{align}
So we know $Z^2+ W_x^2 = \sum_{l=1}^{k}||h||_{\dot{H}^l}^2$, and $(W_0,W_1,\cdots,W_k)$ involves all the $(k+1)$-th partial derivatives except $\nabla^{k+1}_xh$. The next step is to give lower bounds of the terms in Lemma~\ref{lemDsptHk} by $Z$ and $W$, which will prove helpful in the next sections. Here we denote a binomial coefficient by$\begin{pmatrix}
	m \\
	l
\end{pmatrix}= \frac{m!}{l!(m-l)!}$ in the lemma. 

\begin{lem}\label{lemDsptHkE}
	In the context of Lemma \ref{lemDsptHk}, it holds
	\begin{align}
	T_{i,k-i}^A
	& = ||\nabla^i_x \nabla^{k-i+1}_vh||^2 + (k-i)||\nabla^i_x \nabla^{k-i}_vh||^2
	\notag\\
	T_{i,k-i}^B
	& \geq -(k-i) ||\nabla_x^{i+1}\nabla_v^{k-i-1}h ||\cdot ||\nabla_x^{i}\nabla_v^{k-i}h|| \notag\\
	&\quad -\sum\limits_{l=1}^{i}\sum\limits_{l_1=0}^{i-l}
	\begin{pmatrix}
	i-l \\
	l_1
	\end{pmatrix}
	\sqrt{M}\Big(||\nabla_x^{i-l_1-1}\nabla^{k-i+1}_v h|| +
	||\nabla_x^{i-l_1}\nabla^{k-i+1}_v h||\Big)
	|| \nabla_x^{i}\nabla_v^{k-i}h||
	\notag
	\end{align}
	\begin{align}
	T_{mix}^A
	&\geq -2||\nabla_x^{k-1}\nabla_v^2 h|| \cdot ||\nabla^k_x\nabla_v h|| - ||\nabla_x^{k-1}\nabla_v h|| \cdot ||\nabla^k_xh|| \notag\\
	T_{mix}^B
	&\geq ||\nabla^k_xh||^2  -\sum\limits_{l=1}^{k-1}\sum\limits_{l_1=0}^{k-l-1}
	\begin{pmatrix}
	k-1-l \\
	l_1
	\end{pmatrix}
	\sqrt{M}\big(||\nabla_x^{k-l_1-2}\nabla^2_v h||+ ||\nabla_x^{k-l_1-1}\nabla^2_v h||\big)
	|| \nabla^k_x h|| \notag \\
	& -\sum\limits_{l=1}^{k} \sum\limits_{l_1=0}^{k-l}\begin{pmatrix}
	k-l \\
	l_1
	\end{pmatrix}
	||\nabla^{k-1}_x\nabla_v h||\cdot\sqrt{M}
	\big(||\nabla_v\nabla_x^{k-l_1-1} h||+ ||\nabla_v\nabla_x^{k-l_1} h||\big). \notag
	\end{align}
	As a consequence, we have
	\begin{align}
	\label{EqDsptHk1E}
	T_{i,k-i}^A&\geq W_i^2,\\
	\label{EqDsptHk2E}
	T_{k,0}^B
	&\geq -2^{k+1}\sqrt{M}ZW_x - k\sqrt{M}W_kW_x,\\
	\label{EqDsptHk2E'}
	T_{i,k-i}^B
	&\geq
	- k Z^2-ZW_x - 2^i\sqrt{M}(2Z^2 + ZW_i), \mbox{ where } 0\leq i\leq k-1.
	\end{align}	
	\begin{align}
	\label{EqDsptHk3E}
	T_{mix}^A \geq -2W_{k-1}W_k - ZW_x,
	\end{align}
	\begin{align}
	\label{EqDsptHk4E}
	T_{mix}^B\geq W_x^2 -2^{k}\sqrt{M}ZW_x - (k-1)\sqrt{M}W_{k-1}W_x-  2^{k}\sqrt{M}Z(2Z+W_{k}).
	\end{align}	
\end{lem}
\begin{rmq}[A possible refinement] One can replace the inequality \eqref{IneqVk} in Assumption \ref{HypBHess} by 
	\begin{equation}\label{IneqVk*}
		\int |\nabla_x^{2}D_x^{\alpha} V\cdot\nabla_vg|^2 \dd\mu\leq M^* \bigg(\int |\nabla_vg|^2\dd\mu + \int |\nabla^2_{xv}g|^2\dd\mu\bigg)
	\end{equation}
for every $\alpha $ with $0\leq |\alpha|\leq k-1$. Subsequently, one can proceed to work with derivatives such as $D_x^{\alpha}D_v^{\beta}$ rather than those like $\nabla_x^i\nabla^j_v$; the results in Lemma \ref{lemDsptHk} hold under corresponding modifications, and so do the first part of Lemma \ref{lemDsptHkE}. The price is that  one is then led to a larger (but still sparse in some sense) matrix, rather than a matrix of the $Z$ and $W$. 
\end{rmq}

\begin{proof} The inequality \eqref{EqDsptHk1E} follows from \eqref{EqDsptHk1} in Lemma \ref{lemDsptHk},
	\begin{equation}
	T_{i,k-i}^A= ||\nabla^i_x \nabla^{k-i+1}_vh||^2 + (k-i)||\nabla^i_x \nabla^{k-i}_vh||^2 \geq W_i^2
	\end{equation}
	for $0\leq i\leq k$. Note that $(k-i)||\nabla^i_x \nabla^{k-i}_vh||^2(0\leq i\leq k)$ are discarded for the simplification of our presentation, although they are good terms which might help.
	
	Similarly, for the expressions in $T_{i,k-i}^B, T_{mix}^A$ and $T_{mix}^B$, the term $||\nabla^k_x h||^2$ in $T_{mix}^B$ is the only one that really helps in our proof. All the other terms will be bounded from below by applying Cauchy-Schwarz inequality directly.
	
	Due to the (generalized) Leibnitz rule, it holds that
	\begin{equation}
	\nabla_x^{m_1-l}(-\nabla^2V\cdot\nabla_v)_{\alpha}\nabla_x^{l-1}\nabla_v^{m_2}h
	= -\sum\limits_{l_1=0}^{m_1-l} \begin{pmatrix}
	m_1-l \\
	l_1
	\end{pmatrix}
	\sum\limits_{j=1}^{d}\big(\nabla_x^{l_1} \partial^2_{x_{\alpha}x_j}V\big)(\partial_{v_j}
	\nabla_x^{m_1-l_1-1}\nabla^{m_2}_v h) \notag.
	\end{equation}
Substitute it into \eqref{EqDsptHk2}, we therefore have
	\begin{align}\label{4iB1}
	T_{i,k-i}^B
	&= \sum\limits_{l=1}^{k-i}\langle \nabla_x^{i}\nabla_v^{k-i-l}\nabla_x\nabla_v^{l-1}h, \nabla_x^{i}\nabla_v^{k-i}h\rangle \notag\\
	& \quad -\sum\limits_{l=1}^{i}\sum\limits_{l_1=0}^{i-l} \begin{pmatrix}
	i-l \\
	l_1
	\end{pmatrix}
	\langle \sum\limits_{j=1}^{d}\big(\nabla_x^{l_1} \nabla_x\partial_{x_j}V\big)(\partial_{v_j}
	\nabla_x^{i-l_1-1}\nabla^{k-i}_v h), \nabla_x^{i}\nabla_v^{k-i}h\rangle \notag\\
	&\geq -\sum\limits_{l=1}^{k-i} ||\nabla_x^{i}\nabla_v^{k-i-l}\nabla_x\nabla_v^{l-1}h ||\cdot ||\nabla_x^{i}\nabla_v^{k-i}h|| \notag\\
	&\quad -\sum\limits_{l=1}^{i}\sum\limits_{l_1=0}^{i-l} \begin{pmatrix}
	i-l \\
	l_1
	\end{pmatrix}
	||\sum\limits_{j=1}^{d}\big(\nabla_x^{l_1} \nabla_x\partial_{x_j}V\big)\big(\partial_{v_j}
	\nabla_x^{i-l_1-1}\nabla^{k-i}_v h\big)||\cdot|| \nabla_x^{i}\nabla_v^{k-i}h|| \notag\\
	& \geq -(k-i) ||\nabla_x^{i+1}\nabla_v^{k-i-1}h ||\cdot ||\nabla_x^{i}\nabla_v^{k-i}h|| \notag\\
	&\quad -\sum\limits_{l=1}^{i}\sum\limits_{l_1=0}^{i-l}
	\begin{pmatrix}
	i-l \\
	l_1
	\end{pmatrix}
	\sqrt{M}\Big(||\nabla_x^{i-l_1-1}\nabla^{k-i+1}_v h|| +
	||\nabla_x^{i-l_1}\nabla^{k-i+1}_v h||\Big)
	|| \nabla_x^{i}\nabla_v^{k-i}h||
	\end{align}
	where in the last inequality we have applied
	\begin{align*}
	||\sum\limits_{j=1}^{d}\big(\nabla_x^{l_1} \nabla_x\partial_{x_j}V\big)\big(\partial_{v_j}
	\nabla_x^{i-l_1-1}\nabla^{k-i}_v h\big)||
	&\leq \sqrt{M}\Big(||\nabla_{v}
	\nabla_x^{i-l_1-1}\nabla^{k-i}_v h||+ ||\nabla_{v}
	\nabla_x^{i-l_1-1+1}\nabla^{k-i}_v h||\Big)\\
	&= \sqrt{M}\Big(||
	\nabla_x^{i-l_1-1}\nabla^{k-i+1}_v h||+ ||
	\nabla_x^{i-l_1}\nabla^{k-i+1}_v h||\Big)
	\end{align*}
	which holds true by our assumption \eqref{IneqVk}. Note that if $i=0$, the summation in the last line of \eqref{4iB1} is over a empty set, and hence it equals to zero.
	
	We then reformulate the lower bound in terms of $Z, W_x, W_0, \cdots, $ and $W_k$. First we observe that
	\[
	-(k-i) ||\nabla_x^{i+1}\nabla_v^{k-i-1}h || ||\nabla_x^{i}\nabla_v^{k-i}h||
	\geq\left\{\begin{array}{ll}
	0,& \mbox{ if } i=k; \\
	-ZW_x,& \mbox{ if } i=k-1;\\
	-kZ^2,& \mbox{ if } 0\leq i<k-1,
	\end{array}
	\right.
	\]
	and so for $i<k$,
	\begin{equation}\label{4iB1a}
	-(k-i) ||\nabla_x^{i+1}\nabla_v^{k-i-1}h || ||\nabla_x^{i}\nabla_v^{k-i}h||\geq -kZ^2-ZW_x.
	\end{equation}
	
	Next we observe that for $i\geq 1$
	\[
	\Big(||\nabla_x^{i-l_1-1}\nabla^{k-i+1}_v h|| +
	||\nabla_x^{i-l_1}\nabla^{k-i+1}_v h||\Big)
	|| \nabla_x^{i}\nabla_v^{k-i}h||
	\leq\left\{\begin{array}{ll}
	2Z^2,& \mbox{ if } i<k, l_1\geq 1; \\
	(Z+W_{i})Z, &\mbox{ if } i<k, l_1=0;\\
	2ZW_x,& \mbox{ if } i=k, l_1\geq 1; \\
	(Z+W_k)W_x,& \mbox{ if } i=k, l_1=0.
	\end{array}
	\right.
	\]
	
	In fact, $\nabla_x^{i-l_1-1}\nabla^{k-i+1}_v h$ contains derivatives in $v$-direction and its highest order of derivatives of $h$ is $k-l_1$, so it depends only on $Z$ ($W$ is not involved at all). While the highest order of derivatives of $h$ in $||\nabla_x^{i-l_1}\nabla^{k-i+1}_v h||$ is not greater than $k+1$ with equality if and only if $l_1=0$; moreover, in the equality case, $||\nabla_x^{i-l_1}\nabla^{k-i+1}_v h||$ becomes $W_i$. If $l_1\geq 1$, then $i-l_1 + k-i+1\leq k$ and $k-i+1\geq 1$, so it follows that $||\nabla_x^{i-l_1}\nabla^{k-i+1}_v h||$ can be bounded by $Z$. The other factor $|| \nabla_x^{i}\nabla_v^{k-i}h||$  either becomes $W_x$ (if $i=k$), or occurs in $Z$ and thus can be bounded by $Z$ (if $i\leq k-1$).

	Therefore for $1\leq i\leq k-1$ , it holds
	\begin{align}
	&-\sum\limits_{l=1}^{i}\sum\limits_{l_1=0}^{i-l}
	\begin{pmatrix}
	i-l \\
	l_1
	\end{pmatrix}
	\sqrt{M}\Big(||\nabla_x^{i-l_1-1}\nabla^{k-i+1}_v h|| +
	||\nabla_x^{i-l_1}\nabla^{k-i+1}_v h||\Big)
	|| \nabla_x^{i}\nabla_v^{k-i}h||\notag\\
	\geq& -\sum\limits_{l=1}^{i}\sum\limits_{l_1=0}^{i-l}
	\begin{pmatrix}
	i-l \\
	l_1
	\end{pmatrix}
	\sqrt{M}\Big(
	2Z+W_i\Big)Z\notag\\
	\label{4iB1b}
	\geq& -2^{i}\sqrt{M}(2Z^2 + ZW_i).
	\end{align}
	Whereas for $i=k$, it holds
	\begin{align}
	&-\sum\limits_{l=1}^{i}\sum\limits_{l_1=0}^{i-l}
	\begin{pmatrix}
	i-l \\
	l_1
	\end{pmatrix}
	\sqrt{M}\Big(||\nabla_x^{i-l_1-1}\nabla^{k-i+1}_v h|| +
	||\nabla_x^{i-l_1}\nabla^{k-i+1}_v h||\Big)
	|| \nabla_x^{i}\nabla_v^{k-i}h||\notag\\
	\geq& -\bigg\{\sum\limits_{l=1}^{k}\sum\limits_{l_1=1}^{k-l}
	\begin{pmatrix}
	k-l \\
	l_1
	\end{pmatrix}
	\sqrt{M}\cdot 2ZW_x + \sum\limits_{l=1}^{k}\sum\limits_{l_1=0}\begin{pmatrix}
	k-l \\
	l_1
	\end{pmatrix}
	\sqrt{M}(Z+W_k)W_x\bigg\}\notag\\
	\label{4iB1c}
	\geq& -2^{k+1}\sqrt{M}ZW_x -k\sqrt{M}W_kW_x.
	\end{align}
	
	Then \eqref{EqDsptHk2E} and \eqref{EqDsptHk2E'} follows by the lower bounds in \eqref{4iB1}, \eqref{4iB1a}, \eqref{4iB1b} and \eqref{4iB1c}.
	\vspace{0.314cm}

	\textbf{Step 2.} Next we give lower bounds for the mixed terms $T_{mix}^A $ and $T_{mix}^B$, and thus for $T_{mix}^A +T_{mix}^B=\langle \nabla^{k-1}_x\nabla_vLh,\nabla^k_xh\rangle + \langle \nabla^{k-1}_x\nabla_vh,\nabla^k_xLh\rangle$. By \eqref{EqDsptHk3}, we find
	\begin{align}
	T_{mix}^A &\geq -2||\nabla_x^{k-1}\nabla_v^2 h|| \cdot ||\nabla^k_x\nabla_v h|| - ||\nabla_x^{k-1}\nabla_v h|| \cdot ||\nabla^k_xh| \notag\\
	\label{4iiA1}
	&\geq -2W_{k-1}W_k - ZW_x.
	\end{align}
	By \eqref{EqDsptHk4} and the generalized Leibniz rule, we have
	\begin{equation}\label{4iiB0}
	T_{mix}^B= ||\nabla^k_xh||^2 + (\Rom{1}) + (\Rom{2})
	\end{equation}
	where $(\Rom{1}), (\Rom{2})$ are given and bounded from below as follows (again due to \eqref{IneqVk}),
	\begin{align*}
	(\Rom{1})
	&= -\sum\limits_{l=1}^{k-1}\sum\limits_{l_1=0}^{k-l-1}
	\begin{pmatrix}
	k-1-l \\
	l_1
	\end{pmatrix}
	\langle\sum\limits_{j=1}^{d}\big(\nabla_x^{l_1}\nabla_x \partial_{x_j} V\big)\big(\partial_{v_j}\nabla_x^{k-l_1-2}\nabla_v h\big), \nabla^k_x h\rangle\\
	&\geq - \sum\limits_{l=1}^{k-1}\sum\limits_{l_1=0}^{k-l-1}
	\begin{pmatrix}
	k-1-l \\
	l_1
	\end{pmatrix}
	||\sum\limits_{j=1}^{d}\big(\nabla_x^{l_1}\nabla_x \partial_{x_j} V\big)\big(\partial_{v_j}\nabla_x^{k-l_1-2}\nabla_v h\big)||
	\cdot|| \nabla^k_x h||  \\
	&\geq - \sum\limits_{l=1}^{k-1}\sum\limits_{l_1=0}^{k-l-1}
	\begin{pmatrix}
	k-1-l \\
	l_1
	\end{pmatrix}
	\sqrt{M}\big(||\nabla_x^{k-l_1-2}\nabla^2_v h||+ ||\nabla_x^{k-l_1-1}\nabla^2_v h||\big)
	|| \nabla^k_x h||,
	\end{align*}
	and
	\begin{align*}
	(\Rom{2}) 
	&= -\sum\limits_{l=1}^{k} \sum\limits_{l_1=0}^{k-l}\begin{pmatrix}
	k-l \\
	l_1
	\end{pmatrix}\langle \nabla^{k-1}_x\nabla_v h, \sum\limits_{j=1}^{d}\big(\nabla_x^{l_1}\nabla_x\partial_{x_j}V\big)
	\big(\partial_{v_j}\nabla_x^{k-l_1-1} h\big)\rangle  \\
	&\geq -\sum\limits_{l=1}^{k} \sum\limits_{l_1=0}^{k-l}\begin{pmatrix}
	k-l \\
	l_1
	\end{pmatrix}
	||\nabla^{k-1}_x\nabla_v h||\cdot
	||\sum\limits_{j=1}^{d}\big(\nabla_x^{l_1}\nabla_x\partial_{x_j}V\big)
	\big(\partial_{v_j}\nabla_x^{k-l_1-1} h\big)||\\
	&\geq -\sum\limits_{l=1}^{k} \sum\limits_{l_1=0}^{k-l}\begin{pmatrix}
	k-l \\
	l_1
	\end{pmatrix}
	||\nabla^{k-1}_x\nabla_v h||\cdot\sqrt{M}
	\big(||\nabla_v\nabla_x^{k-l_1-1} h||+ ||\nabla_v\nabla_x^{k-l_1} h||\big).
	\end{align*}

	Now we give lower bounds of $(\Rom{1})$ and $(\Rom{2})$ in terms of $Z$ and $W$. Note that
	\[
	||\nabla_x^{k-l_1-2}\nabla^2_v h||+ ||\nabla_x^{k-l_1-1}\nabla^2_v h||
	\leq\left\{\begin{array}{ll}
	2Z,& \mbox{ if } l_1\geq 1; \\
	Z+W_{k-1}, &\mbox{ if } l_1=0.
	\end{array}
	\right.
	\]
	Hence we have
	\begin{align}
	(\Rom{1})&\geq - \sum\limits_{l=1}^{k-1}\sum\limits_{l_1=0}^{k-l-1}
	\begin{pmatrix}
	k-1-l \\
	l_1
	\end{pmatrix}
	\sqrt{M}\cdot 2ZW_x - \sum\limits_{l=1}^{k-1}\sum\limits_{l_1=0}
	\begin{pmatrix}
	k-1-l \\
	l_1
	\end{pmatrix}
	\sqrt{M}W_{k-1}W_x \notag\\
	\label{4iiB1}
	&\geq -2^{k}\sqrt{M}ZW_x - (k-1)\sqrt{M}W_{k-1}W_x.
	\end{align}
	Similarly, for $(\Rom{2})$, since
	\[
	||\nabla_v\nabla_x^{k-l_1-1} h||+ ||\nabla_v\nabla_x^{k-l_1} h||
	\leq\left\{\begin{array}{ll}
	2Z,& \mbox{ if } l_1\geq 1; \\
	Z+W_{k}, &\mbox{ if } l_1=0,
	\end{array}
	\right.
	\]
	we obtain
	\begin{align}
	(\Rom{2})
	&\geq -\sum\limits_{l=1}^{k} \sum\limits_{l_1=0}^{k-l}\begin{pmatrix}
	k-l \\
	l_1
	\end{pmatrix}
	Z\cdot\sqrt{M}
	\big(2Z+W_{k}\big).\notag\\
	\label{4iiB2}
	&\geq -  2^{k}\sqrt{M}Z(2Z+W_{k}).
	\end{align}
	Therefore we can deduce the inequality \eqref{EqDsptHk4E}.
\end{proof}

\section{Hypocoercivity in $H^k(\mu)$}
\label{SectCoer}
Now we are ready to prove  Theorem \ref{thmHigh}. The following coercive estimate is the essential result to this end. The key of the proof is to construct an auxiliary Sobolev norm. Let us recall that the $H^l(\mu)$-seminorm $||\cdot||_{\dot{H}^l}$ is defined by
\[
||h||_{\dot{H}^l}^2 := \sum\limits_{i=0}^{ l}||\nabla_x^i\nabla_v^{l-i}h||_{L^2(\mu)}^2
= \sum\limits_{i=0}^{ l}\sum\limits_{|\alpha|= i, |\beta|=l-i} \int |D_x^{\alpha}D_v^{\beta}h|^2\dd\mu
\]
where $\alpha,\beta$ are multi-indexes.

\begin{prop}\label{PropCoerHk}
	Under the assumptions in Theorem \ref{thmHigh}. There exists a twisted $H^k(\mu)$-norm, denoted by $((\cdot,\cdot))^{\frac12}_{H^k}$, which is equivalent to the usual $H^k(\mu)$-norm and satisfies an estimate
	\begin{equation}\label{EqCoerHk}
	((h,Lh))_{H^k}\geq \lambda_{k,0} (\sum\limits_{1\leq l \leq k}||h||_{\dot{H}^l}^2 + \sum\limits_{0\leq l\leq k}||\nabla_x^l\nabla^{k+1-l}_v h||^2)
	\end{equation}
	for some constant $\lambda_{k,0}>0$ and for any rapidly decreasing function $h$ (i.e. $h\in \mathcal{S}(\RR^{2d})$). Here $\lambda_{k,0}$ depends only on $k$ and $M$. As a consequence, it holds the following coercive estimate
	\begin{equation}\label{EqCoerHk*}
	((h,Lh))_{H^k} \geq \lambda_k ((h-\int h\dd\mu,h-\int h\dd\mu))_{H^k}
	\end{equation}
	for some constant $\lambda_k>0$ and for all function $h\in \mathcal{S}(\RR^{2d})$.
\end{prop}
\begin{rmq}
    Note also that the only difference between
	$\sum_{l=0}^{ k}||\nabla_x^l\nabla^{k+1-l}_v h||^2$ and $||h||_{\dot{H}^{k+1}}^2$ is that the former expression does not contain the term $||\nabla^{k+1}_xh||^2$ while the latter one does.
\end{rmq}

Theorem \ref{thmHigh} is a direct corollary of the proposition above.
\begin{proof}[Proof of Theorem \ref{thmHigh}]
	It follows from the coercive estimate in Proposition \ref{PropCoerHk}. Note that as a solution to the kinetic Fokker-Planck equation, $h_t$ is smooth at positive time $t>0$. By Proposition \ref{PropCoerHk} and a standard approximation procedure, we know that at positive time
	\[
	((h,Lh))_{H^k(\mu)}\geq \lambda ((h-\int h\dd\mu, h- \int h\dd\mu))_{H^k(\mu)}.
	\]
	Consequently Gronwall's lemma implies that
	\[
	((h-\int h\dd\mu, h- \int h\dd\mu))_{H^k(\mu)}
	\leq
	e^{-2\lambda t}((h_0-\int h_0\dd\mu, h_0- \int h_0\dd\mu))_{H^k(\mu)}
	\]
	Note that  the norm induced by $((\cdot, \cdot))_{H^k(\mu)}$ is equivalent to the usual $H^k(\mu)$ norm with explicit constants. The theorem then follows.
\end{proof}

Define the twisted $H^l(\mu)$-seminorms $((h,h))^{\frac12}_{\dot{H}^l}$ by
\begin{equation}\label{EqTwHkSemi}
((h,h))_{\dot{H}^l}
:= \sum\limits_{0\leq i \leq l} \omega_{l,i}||\nabla^{l-i}_v\nabla^i_x h||^2 + 2\omega_{l} \langle \nabla_x^{l-1}\nabla_v h, \nabla_x^{l}h\rangle
\end{equation}
with all the constants $\omega_{l,i} (0\leq i\leq l)$, $\omega_{l}$ being strictly positive and satisfying
\begin{equation}
\omega_{l}^2<\omega_{l,l-1}\omega_{l,l}.
\end{equation}
It is then clear that such a seminorm is equivalent to the usual $H^l(\mu)$-seminorm in the sense that there exist constants $c_1=c_1(l),c_2=c_2(l) >0$ such that
\[
c_1 ||h||_{\dot{H}^{l}}^2 \leq ((h,h))_{\dot{H}^l} \leq c_2 ||h||_{\dot{H}^{l}}^2.
\]
For example, one may have a look at the case $l=1$ or $l=2$:
\begin{align*}
((h,h))_{\dot{H}^1}&:= \omega_{1,0}||\nabla_v h||^2 + \omega_{1,1}||\nabla_xh||^2 + 2\omega_1 \langle \nabla_v h, \nabla_xh\rangle,\\
((h,h))_{\dot{H}^2}&:= \omega_{2,0}||\nabla_v^2h||^2 + \omega_{2,1} ||\nabla_{xv}^2h||^2 +\omega_{2,2}||\nabla_x^2h||^2 + 2 \omega_{2}\langle\nabla_{xv}^2h ,\nabla_x^2 h\rangle.
\end{align*}
We can take $\omega_{1,0}=a$, $\omega_{1,1}=c$ and $\omega_{1}=b$ with $a,b,c$ given in Villani's auxiliary $H^1$ norm. 

We then define the twisted $H^k(\mu)$-norm by
\begin{align}\label{EqTwHk}
((h,h))_{H^k}&:= ||h||^2 + \sum\limits_{l=1}^{k}((h,h))_{\dot{H}^l}\\
&= \sum\limits_{0\leq i\leq j\leq k}
\omega_{j,i}||\nabla^{j-i}_v\nabla^i_x h||^2
+ \sum\limits_{1\leq j\leq k}2\omega_{j}
\langle \nabla_x^{j-1}\nabla_v h, \nabla_x^{j}h\rangle
\end{align}
with suitable coefficients $\omega_j$ ($1\leq j\leq k$), $\omega_{j,i}$($0\leq i\leq j\leq k$) to be determined later (by an induction argument on $k$). Here we set $\omega_{0,0}=1$ as a convention.

\begin{proof}[Proof of Proposition \ref{PropCoerHk}]
	We prove it by induction on $k$. The proposition for $k=1$ is verified by the construction in \cite{Villani}. Or we can start with the induction basis by $k=0$: with $\omega_{0,0}=1$ and the usual $L^2(\mu)$-norm, we have $\lambda_{0,0}=1$ since
	\[
	\langle h, Lh\rangle = ||\nabla_v h||^2.
	\]
	Now consider $k\geq 1$. The induction hypothesis asserts that the coercive estimate holds true for $k-1$, i.e., there exists a seminorm $((\cdot,\cdot))^{\frac12}_{H^{k-1}}$ and a constant $\lambda_{k-1,0}>0$ (depending only on $k-1$ and $M$) such that
	\begin{equation}\label{4-1}
	((h,Lh))_{H^{k-1}} \geq \lambda_{k-1,0} (\sum\limits_{1\leq l \leq k-1}||h||_{\dot{H}^l}^2 + \sum\limits_{0\leq l\leq k-1}||\nabla_x^l\nabla^{k-l}_v h||^2).
	\end{equation}
	We shall prove the existence of $\omega_{k}, \omega_{k,0},\omega_{k,1},\cdots,\omega_{k,k}$ in \eqref{EqTwHkSemi} such that the norm defined by
	\[
	((h,Lh))_{H^{k}}= ((h,Lh))_{H^{k-1}} + ((h,Lh))_{\dot{H}^{k}}
	\]
	satisfies a coercive estimate
	\begin{equation}\label{4-2}
	((h,Lh))_{H^{k}} \geq \lambda_{k,0} (\sum\limits_{1\leq l \leq k}||h||_{\dot{H}^l}^2 + \sum\limits_{0\leq l\leq k}||\nabla_x^l\nabla^{k+1-l}_v h||^2)
	\end{equation}
	for some $\lambda_{k,0}>0$ which depends only on $k$ and $M$.
	
	Now we rephrase \eqref{4-1} and \eqref{4-2} in terms of $Z$ and $W$. For convenience we recall that
	\[
	Z:=\big(\sum\limits_{1\leq l \leq k-1}||h||_{\dot{H}^l}^2 + \sum\limits_{0\leq l\leq k-1}||\nabla_x^l\nabla^{k-l}_v h||^2\big)^{\frac12}
	\]
	and $W:=(W_x,W_0,W_1,\cdots, W_{k})^{\mathsf{T}}\in\RR^{k+2}$ where
	\[
	W_x= ||\nabla_x^{k}h||,\quad W_l=||\nabla_x^l\nabla^{k+1-l}_v h||, \quad 0\leq l\leq k.
	\]
	So we have $Z^2+ W_x^2 = \sum_{l=1}^{k}||h||_{\dot{H}^l}^2$. Then the induction hypothesis \eqref{4-1} is equivalent to
	\[
	((h,Lh))_{H^{k-1}}\geq \lambda_{k-1,0}Z^2,
	\]
	and the desired estimate \eqref{4-2} is equivalent to
	\[((h,Lh))_{H^{k}}\geq \lambda_{k,0}(Z^2 + |W|^2).
	\]
	The idea here is to distinguish, on the one hand, the derivatives of order not greater than $k$ except $\nabla^k_x$, and on the other hand, $\nabla^k_x$ and the derivatives of order $k+1$ except $\nabla^{k+1}_x$. The former collection of derivatives already appeared in the coercive estimate of $((h,Lh))_{{H}^{k-1}}$, while the latter one is the new ones coming from $((h,Lh))_{\dot{H}^k}$. Such a division will prove helpful in the induction procedure .
	
	Then we shall bound $((h,Lh))_{\dot{H}^k}$ from below in terms of $Z$ and $W$, more precisely, we shall prove
	\begin{align}\label{EqHkGoal}
	((h,Lh))_{\dot{H}^k} &\geq \omega_k(-K_1 Z^2-K_2 Z|W| + \eta |W|^2)
	\end{align}
	for some constants  $K_1,K_2$, and $\eta= \frac{1}{64k ^2M} $. Later in the proof, we shall see that $\omega_k$ can be chosen as small as one desires (without any modification of $K_1,K_2$ and $\eta$), and so that we are able to 
	obtain a coercive estimate in form of \eqref{4-2}.
	\vspace{0.314cm}
	
	By the definition \eqref{EqTwHkSemi}, we have
	\begin{align*}
	((h,Lh))_{\dot{H}^k}
	&:= \sum\limits_{0\leq i \leq k} \omega_{k,i}\langle\nabla^{k-i}_v\nabla^i_x Lh, \nabla^{k-i}_v\nabla^i_x h\rangle \\
	&\quad+ \omega_{k} \big(\langle \nabla_x^{k-1}\nabla_v L h, \nabla_x^{k}h\rangle + \langle \nabla_x^{k-1}\nabla_v h, \nabla_x^{k}Lh\rangle\big)\\
	& = \sum\limits_{0\leq i \leq k} \omega_{k,i}(T_{i,k-i}^A + T_{i,k-i}^B ) + \omega_{k}(T_{mix}^A + T_{mix}^B)
	\end{align*}
	where $T_{i,k-i}^A, T_{i,k-i}^B,T_{mix}^A,T_{mix}^B$ are defined in Section \ref{SectPre}(see Lemma \ref{lemDsptHk} for instance).
	
	\textbf{Step 1.} First of all, let us set some relation between the coefficients in $((h,h))_{\dot{H}^k}$. We may assume $M\geq 1$. Set
	\[
	\omega_{k,i}= \left\{\begin{aligned}64k^2M\omega_k, &\mbox{ for } 0\leq i\leq  k-1; \\  \frac{1}{16k^2M}\omega_{k}, &\mbox{ for } i=k,
	\end{aligned}\right.
	\]
	where $\omega_k$ will be determined later. Then by Lemma \ref{LemApp1} in the appendix, we know that in the sense of quadratic forms
	\begin{align}\label{EqCoerSwHk0}
	\begin{pmatrix}
	\omega_k         & 0            &0\\
	-k\omega_k\sqrt{M}      & \omega_{k,k-1}          &0\\
	-2k\omega_{k,k}\sqrt{M}       & -2\omega_k   &\omega_{k,k}
	\end{pmatrix}
	&\geq\text{Diag}(\frac{\omega_k}{2} ,\frac{\omega_{k,k-1}}{4},\frac{\omega_{k,k}}{4}).
	\end{align}
	It follows that
	\begin{align}
	Q(W)&:=\omega_{k}W_x^2 + \sum_{i=0}^{k}\omega_{k,i}W_i^2 - (k-1)\omega_{k}\sqrt{M}W_xW_{k-1}- k\omega_{k,k}\sqrt{M}W_xW_{k} - 2\omega_kW_{k-1}W_k \notag\\
	\label{IneqQW}
	&\geq \frac{1}{4}\big(\omega_k W_x^2+ \sum_{i=0}^{k}\omega_{k,i}W_i^2\big).
	\end{align}
	In particular, it holds
	\begin{equation}\label{EqCoerSwHk}
	Q(W)\geq \frac14 \cdot \frac{\omega_k}{16k^2M}|W|^2 = \eta \omega_k|W|^2.
	\end{equation}
	where $\eta = \frac{1}{64k^2M}$ as claimed in \eqref{EqHkGoal}.
	\vspace{0.314cm}
	
	\textbf{Step 2.} We then apply the estimates in Lemma \ref{lemDsptHkE} to bound $((h,Lh))_{\dot{H}^k}$ from below. By the estimates \eqref{EqDsptHk1E},\eqref{EqDsptHk2E} and \eqref{EqDsptHk2E'}, it holds
	\begin{align}
	\sum\limits_{i=0}^{k}\omega_{k,i}(T_{i,k-i}^A + T_{i,k-i}^B)
	&\geq \sum\limits_{i=0}^{k}\omega_{k,i}W_i^2 -  \sum\limits_{i=0}^{k-1}\omega_{k,i}\left[kZ^2 + ZW_x +2^i\sqrt{M}(2Z^2 + ZW_i) \right] \notag\\
	&\quad -2^{k+1}\sqrt{M}\omega_{k,k}ZW_x - k\sqrt{M}\omega_{k,k}W_kW_x.  \notag\\
	&\geq \sum\limits_{i=0}^{k}\omega_{k,i}W_i^2  -K'_1 \omega_kZ^2 -K'_2\omega_k Z|W|- k\sqrt{M}\omega_{k,k}W_kW_x \notag
	\end{align}
	for some constants $K'_{1}$ and $K'_2$. We stress that the constants $K'_{1}$ and $K'_2$  depend only on $k$ and $M$, since the ratio $\omega_{k,i}/\omega_k$ is a constant depending only on $k$ and $M$ (independent of $\omega_k$).
	
	Similarly, it follows from \eqref{EqDsptHk3E} and \eqref{EqDsptHk4E} that
	\begin{align}
	&\quad\quad  T_{mix}^A + T_{mix}^B \notag\\
	&\geq -2W_{k-1}W_k - ZW_x + W_x^2-2^{k}\sqrt{M}ZW_x - (k-1)\sqrt{M}W_{k-1}W_x-  2^{k}\sqrt{M}Z(2Z+W_{k})\notag\\
	&\geq W_x^2-2W_{k-1}W_k- (k-1)\sqrt{M}W_{k-1}W_x - K_1''Z^2 - K_2''Z|W|\notag
	\end{align}
	where the constants $K_1''$ and $K_2''$ depend only on $k$ and $M$.
	
	We now come to conclude that the desired estimate \eqref{EqHkGoal} holds true for some constants $K_1$ and $K_2$ given by
	\[
	K_1=K'_1+ K_1'',\quad K_2=K'_2+ K_2''
	\]
	which depend only on $k,M$. And it is clear that the constants can be explicitly computed. Indeed,
	\begin{align}
	&\quad ((h,Lh))_{\dot{H}^k}
	= \sum\limits_{0\leq i \leq k} \omega_{k,i}(T_{i,k-i}^A + T_{i,k-i}^B ) + \omega_{k}(T_{mix}^A + T_{mix}^B)\notag\\
	&\geq \sum\limits_{i=0}^{k}\omega_{k,i}W_i^2  -K'_1 \omega_kZ^2 -K'_2\omega_k Z|W|- k\sqrt{M}\omega_{k,k}W_kW_x\notag\\
	& \quad + \omega_k\left[W_x^2-2W_{k-1}W_k- (k-1)\sqrt{M}W_{k-1}W_x - K_1''Z^2 - K_2''Z|W|\right]\notag\\
	&=Q(W)- \omega_k K_1Z^2 - \omega_k K_2 Z|W| \notag\\
	&\geq \omega_k\big\{\eta |W|^2 - K_1Z^2 -  K_2 Z|W|\big\} \notag
	\end{align}
	where $Q(W)$ was introduced in \eqref{IneqQW} and it satisfies \eqref{EqCoerSwHk}, c.f. Step 1.  By the induction hypothesis,
	\[
	((h,Lh))_{H^{k-1}}\geq \lambda_{k-1,0} Z^2,
	\]
	and so we obtain
	\begin{align*}
	((h,Lh))_{H^{k}}&=((h,Lh))_{H^{k-1}}+ ((h,Lh))_{\dot{H}^{k}}\\
	&\geq \lambda_{k-1,0} Z^2 + \omega_k\big\{\eta |W|^2 - K_1Z^2 -  K_2 Z|W|\big\}.
	\end{align*}
	Then it holds that
	\[
	((h,Lh))_{H^{k}}\geq \lambda_{k,0}(Z^2 + W^2)
	\]
	where
	\[
	\omega_k= \min\bigg\{\frac{\lambda_{k-1,0}}{2K_1}, \frac{3\eta\lambda_{k-1,0}}{4K_2^2}\bigg\}, \quad
	\lambda_{k,0}= \min\bigg\{ \frac{\lambda_{k-1,0}}{4}, \frac{\omega_k \eta}{4}\bigg\}.
	\]
	By now, the proof of \eqref{EqCoerHk} is thus finished.
	\vspace{0.314cm}
	
	\textbf{Step 3.} To prove its consequence \eqref{EqCoerHk*}, it suffices to observe that the constructed twisted $H^{l}(\mu)$-seminorm associated to $ ((\cdot,\cdot))_{\dot{H}^{l}}$ is bounded by (in fact, equivalent to) the usual $H^{l}(\mu)$-seminorm up to a constant, for each $1\leq l\leq ~k$. Indeed, in the setting of Step 1, we may find $\omega_{l,l}\omega_{l,l-1}=4\omega_{l}^2$ for each $1\leq l\leq k$, and then
	\[
	2\omega_{l}|\langle \nabla^{l-1}_x\nabla_v h, \nabla_x^l h\rangle| 
	\leq \frac{1}{2}\big(\omega_{l,l-1}||\nabla^{l-1}_x\nabla_v h||^2+
	\omega_{l,l}||\nabla_x^l h||^2\big),
	\]
	which follows that the seminorm defined by \eqref{EqTwHkSemi} satisfies
	\[
	((h,h))_{\dot{H}^{l}}\leq \frac32 \max\{\omega_{l,l-1}, \omega_{l,l}\} ||h||_{\dot{H}^l}^2.
	\]
	 By the tensorisation property of Poincar\'e inequality (c.f. \cite[Proposition 4.3.1]{BGL14}),
	\[
	||h-\int h\dd\mu||^2 \leq ||\nabla_vh||^2 + \kappa||\nabla_x h||^2
	\]
    where $\kappa$ is given in the Poincar\'e inequality in the assumption \ref{HypPI}.
	
	Gathering the above inequalities, we obtain that
	\[
	((h-\int h\dd\mu,h-\int h\dd\mu))_{{H}^{k}} \leq C (Z^2+ W_x^2)
	\]
	with $C= \max\{1, \kappa, \frac32 \omega_{l,l-1}, \frac32 \omega_{l,l} \,|\, 1\leq l\leq k\}$. Finally we conclude by \eqref{EqCoerHk} that
	\[
	((h,Lh))_{H^{k}}\geq \frac{\lambda_{k,0}}{C}((h-\int h\dd\mu,h-\int h\dd\mu))_{{H}^{k}}
	\]
	which completes the proof of \eqref{EqCoerHk*}.
	
\end{proof}

\begin{rmq}
	Although we have no attempt to give a sharp rate of convergence, we comment on possible refinements of the constants and thus the rate of convergence in the proof of Proposition \ref{PropCoerHk}. The first possibility is that, just as the inequality \eqref{IneqQW} we have proved  for $Q(W)$, we can also find coefficients $\omega_k, \omega_{k,i}(0\leq i\leq k)$ such that
	\[
	((h,Lh))_{\dot{H}^k}\geq \delta\big(\omega_kW_x^2 + \sum\nolimits_{i=0}^{k}\omega_{k,i}W_i^2\big)
	\]
	holds for some $\delta>0$ and for all $h\in \mathcal{S}(\RR^{2d})$. Another possibility is to refine the lower bounds in terms of $Z$ in the proof, for instance, we may distinguish $||\nabla_v h||$, $||\nabla_vh||$, $||\nabla_v^2h||$, $\cdots$, $||\nabla_v^{k}h||$, $||\nabla^{k-1}_v\nabla_x h||$, $\cdots$, $||\nabla^v\nabla_x^{k-1}h||$ (i.e. roughly all the terms appearing in $Z$). It may result in a matrix of very large size. But it is still possible (although technically complicated) to find coefficients such that a coercive estimate holds. The rate of convergence would be better especially when $k$ or $M$ becomes very large.
\end{rmq}

\section{Global hypoellipticity in $H^k(\mu)$}
\label{SectElli}
Guided by \cite[Section A.21.2]{Villani}, we prove the global hypoelliptic estimates for higher derivatives. We closely follow a method due to F.~H\'erau which was used to prove global estimates in $H^1(\mu)$ in short time by constructing a special Lyapunov functional.

To illustrate the efficiency of H\'erau's method, we start by the heat equation. Let $f= f(t,x)$ solve the heat equation in $\RR^d_x$ with $L^2(\dd x)$ initial datum, i.e. $\partial_t f = \Delta f$. For our purpose, we assume that $f$ is smooth and well-behaved at infinity. Then the following a~priori estimates hold,
\[
\frac{\dd}{\dd t} \int f^2\dd x \leq -2 \int |\nabla f|^2 \dd x, \quad \frac{\dd}{\dd t} \int |\nabla f|^2\dd x \leq -2 \int |\nabla^2 f|^2 \dd x.
\]
Hence for the functional
\[
\mathcal{F}(t,f(t,\cdot)):= \int f^2\dd x + 2t \int |\nabla f|^2\dd x,
\]
it holds by the a priori estimates that
\[
\frac{\dd}{\dd t}\mathcal{F}(t,f)\leq -2t \int |\nabla^2 f|^2 \dd x \leq 0.
\]
As a consequence, the functional $\mathcal{F}$ is non-increasing in time. In particular, we know that
\[
2t \int |\nabla f(t,x)|^2\dd x \leq \int  f(0,x)^2\dd x, \quad \mbox{ for } t>0.
\]
By standard approximation, this a priori estimate holds for all solutions of the heat equation with $L^2$ initial datum. It express that $\int |\nabla f(t,x)|^2\dd x $ is of order $t^{-1}$ and, in particular, $f(t,\cdot)\in H^1(\dd x)$ for any $t>0$, provided that $f(0,\cdot)\in L^2(\dd x)$. Note also that it implies higher order global regularity, i.e.
\[
||f(t,\cdot)||^2_{\dot{H}^k(\dd x)} \leq C t^{-k}||f(0,\cdot)||^2_{L^2(\dd x)}, \quad \mbox{ for } t>0.
\]
Indeed, this follows from a simple iteration since the derivatives of $f$ also satisfy the very same heat equation.
\medskip

Now we explain how H\'erau's method can be employed to demonstrate regularity of higher order derivatives for the solutions to \eqref{Eq-kFP*}. The crucial point of his method is to assign a carefully-chosen time-dependent coefficient to each derivative. Following this spirit, we define the functional
\begin{equation}\label{EqHe}
\mathcal{F}_l(t,h_t)
:= \sum\limits_{0\leq i \leq l} \sigma_{l,i}t^{l+2i}||\nabla^{l-i}_v\nabla^i_x h_t||^2 + 2\sigma_{l}t^{3l-1} \langle \nabla_x^{l-1}\nabla_v h_t, \nabla_x^{l}h_t\rangle
\end{equation}
where the to-be-determined coefficients $\sigma_{l,i}, \sigma_l$ ($0\leq i\leq l, 1\leq l$) are strictly positive constants and
\begin{equation}
\sigma_{l,l-1}\sigma_{l,l}> \sigma_l^2.
\end{equation}
(Note that the preceding inequality implies that $t^{l+2i}||\nabla^{l-i}_v\nabla^i_x h_t||^2$ is bounded by $\mathcal{F}_l(t,h_t)$ up to certain constant, see the proof below.) We also set $\sigma_{0,0}:=1$ and
\begin{equation}
\mathcal{F}_{0}(t,h_t):= ||h_t||^2.
\end{equation}
It is a direct computation to show that
\[
-\frac{\dd}{\dd t}\mathcal{F}_{0}(t,h_t)= -\frac{\dd}{\dd t}||h_t||^2 = 2||\nabla_vh_t||^2.
\]
In other words, $\mathcal{F}_{0}$ is a Lyapunov functional for the evolution \eqref{Eq-kFP*}. We can also write
\[
\mathcal{F}_1(t,h_t)= \sigma_{1,0}t||\nabla_vh_t||^2 + \sigma_{1,1}t^{3}||\nabla_xh_t||^2 + 2\sigma_1 t^2\langle \nabla_v h_t, \nabla_x h_t\rangle
\]
and then the Lyapunov functional constructed by H\'erau is in the form of $\mathcal{F}_0(t,h_t)+ \mathcal{F}_1(t,h_t)$ with suitable coefficients.

In the sequel, we shall construct the coefficients of  $\mathcal{F}_{l}$($1\leq l\leq k$) by induction, and then $\sum\nolimits_{l=0}^{k}\mathcal{F}_{l}(t,h_t)$ will be the the Lyapunov functional designed for the hypoellipticity estimates in $H^k(\mu)$.  We shall also prove a strengthened version of the monotonicity along the line, namely,

\begin{prop}\label{propReg}
	In the context of Theorem \ref{thmReg}, there exist strictly positive constants $\sigma_{l,i}, \sigma_l$ ($0\leq i\leq l, 1\leq l\leq k$), defining $\mathcal{F}_l$ ($1\leq l\leq k$) as in \eqref{EqHe}, such that for any solution $h_t=h(t,\cdot)$ to the kinetic Fokker-Planck equation \eqref{Eq-kFP*}, for $0\leq t \leq 1$,
	\begin{equation}\label{EstHe}
	-\frac{\dd}{\dd t}\sum_{l=0}^{k}\mathcal{F}_{l}(t,h_t)
	\geq \frac{\Lambda_{k}}{t}\left( \sum_{1\leq l\leq k}
	 t^{3l}||\nabla_x^lh_t||^2+\sum\limits_{0\leq i \leq l\leq k} t^{l+1+2i}||\nabla^{l-i+1}_v\nabla^i_x h_t||^2\right)
	\end{equation}
	where $\Lambda_k>0$ is some constant depending only on $k$ and $M$. Moreover, the coefficients $\sigma_{l,i}, \sigma_l$ ($0\leq i\leq l, 1\leq l\leq k$) depend only on $l$ and $M$, and
	\[
	\sigma_{l,l-1}\sigma_{l,l}\geq 4\sigma_l^2, \quad \forall 1\leq l\leq k.
	\]
\end{prop}

With this proposition at hand, it is easy to prove the counterparts of H\'erau-Villani's global hypoellipticity estimates in $H^k(\mu)$.

\begin{proof}[Proof of Theorem \ref{thmReg}]
By Proposition \ref{propReg}, $\sum\nolimits_{l=0}^{k}\mathcal{F}_{l}(t,h_t)$ is monotonic increasing in time, and so
\[
\sum\nolimits_{l=0}^{k}\mathcal{F}_{l}(t,h_t)
\leq \mathcal{F}_{0}(0,h_0)= ||h_0||^2.
\]

Note that  $\sigma_{l,l-1}\sigma_{l,l}\geq 4 \sigma_l^2$ ($1\leq l\leq k$) implies
\[
 2\sigma_{l}t^{3l-1} \left|\langle \nabla_x^{l-1}\nabla_v h_t, \nabla_x^{l}h_t\rangle\right|
 \leq \frac12 \left[ \sigma_{l,l-1}t^{3l-2}||\nabla^{l-1}_x\nabla_v h_t||^2 + \sigma_{l,l}t^{3l}||\nabla^l_x h_t||^2 \right]
\]
and hence for $1\leq l\leq k$,
\[
 \frac12\sum\limits_{0\leq i \leq l} \sigma_{l,i}t^{l+2i}||\nabla^{l-i}_v\nabla^i_x h_t||^2
 \leq \mathcal{F}_l(t,h_t).
\]
Therefore, we conclude that for $0\leq t\leq 1$
\begin{equation}
||h_t||^2+\frac12\sum\limits_{l=1}^{k}\sum\limits_{0\leq i \leq l} \sigma_{l,i}t^{l+2i}||\nabla^{l-i}_v\nabla^i_x h_t||^2
\leq ||h_0||^2
\end{equation}
which completes the proof.
\end{proof}

Now we turn to
\begin{proof}[Proof of Proposition \ref{propReg}]
Hereafter we assume that $0< t\leq 1$. The coefficients $\sigma_{l,i}, \sigma_l$ ($0\leq i\leq l\leq k$) will be constructed by induction on $k$. For $k=0$, we choose $\sigma_{0,0}=1$ and then $\Lambda_0 =2$, since
\[
-\frac{\dd}{\dd t}\mathcal{F}_{0}(t,h_t)= 2||\nabla_vh_t||^2.
\]
Now suppose the proposition holds true for $k-1$ with $k\geq 1$. We divide the construction of $\sigma_{k,i}, \sigma_k$ ($0\leq i\leq k$) and $\Lambda_{k}$ into several steps.

We start by a reformulation which might help to simplify the presentation, especially in the estimates below. Define $\mathcal{Z}_{i+j,i}>0$ by
\begin{equation}
\mathcal{Z}_{i+j,i}^2 := t^{j+3i} ||\nabla^{j}_v\nabla^i_x h_t||^2.
\end{equation}
We also write
\begin{equation}
\mathcal{Z}^2:= \sum_{1\leq l\leq k-1}
\mathcal{Z}_{l,l}^2 + \sum\limits_{0\leq i \leq l\leq k-1} \mathcal{Z}_{l+1,i}^2.
\end{equation}
(For example, in the case of $k=1$, $\mathcal{Z}^2= t||\nabla_vh||^2$; in the case of $k=2$, $\mathcal{Z}^2= t||\nabla_vh||^2 + t^3||\nabla_xh||^2 + t^2||\nabla_v^2h||^2 + t^4 ||\nabla_{xv}^2h||^2$.)

Then the desired estimate \eqref{EstHe} becomes
\begin{equation}
-\frac{\dd}{\dd t}\sum_{l=0}^{k}\mathcal{F}_{l}(t,h_t)
\geq \frac{\Lambda_{k}}{t}
\left(\mathcal{Z}^2 +  \mathcal{Z}_{k,k}^2 + \sum\limits_{0\leq i \leq k} \mathcal{Z}_{k+1,i}^2 \right),
\end{equation}
while the induction hypothesis asserts that
\begin{equation}
-\frac{\dd}{\dd t}\sum_{l=0}^{k-1}\mathcal{F}_{l}(t,h_t)
\geq \frac{\Lambda_{k-1}}{t} \mathcal{Z}^2
\end{equation} with some constant $\Lambda_{k-1}$ depending only on $k-1$ and $M$. To finish the induction, it suffices to give a lower bound for the temporal derivative for $\mathcal{F}_k(t,h_t)$. The rough idea is that the induction hypothesis will help when we choose the coefficients $\sigma_{k}, \sigma_{k,i}$ being relatively small. 	

{\bf Step 1.} We may assume $M\geq 1$. We set the relations between $\sigma_{k,i}, \sigma_k$ ($0\leq i\leq k$),
\begin{align}\label{EqCoefSigma}
\sigma_{k,i}= \left\{\begin{array}{ll}
64k^2M\sigma_k, &\mbox{ for } 0\leq i\leq  k-1; \\  \frac{1}{16k^2M}\sigma_{k}, &\mbox{ for } i=k
\end{array}\right.
\end{align}
with $\sigma_k$ to-be-determined. Then by Lemma \ref{LemApp1} in the appendix,
\[
\begin{pmatrix}
\sigma_k    & 0       & 0      \\
-k\sqrt{M}\sigma_k  &\sigma_{k,k-1}    &0\\
-2k\sqrt{M}\sigma_{k,k}   & -2\sigma_k  &\sigma_{k,k}
\end{pmatrix}
\geq
\text{Diag}(\frac12 \sigma_k,  \frac14 \sigma_{k,k-1}, \frac14 \sigma_{k,k})
\]
in the sense of quadratic forms. Note that $\frac32 k\sigma_{k,k}\leq \frac14 \sigma_k$, then
\begin{align}
\mathcal{Q}
&:= -\frac{3}{2} k\sigma_{k,k}\mathcal{Z}_{k,k}^2+ \sum\limits_{i=0}^{k}
\sigma_{k,i}\mathcal{Z}_{k+1,i}^2- k\sqrt{M}\sigma_{k,k}\mathcal{Z}_{k+1,k}\mathcal{Z}_{k,k} -2\sigma_k\mathcal{Z}_{k+1,k-1}\mathcal{Z}_{k+1,k}\notag\\
&\quad\quad \label{EstHeQ}
+\sigma_k \mathcal{Z}_{k,k}^2  -(k-1)\sigma_k\sqrt{M}\mathcal{Z}_{k+1,k-1}\mathcal{Z}_{k,k}\\
&\geq (\sigma_k-\frac{3}{2} k\sigma_{k,k})\mathcal{Z}_{k,k}^2+ \sum\limits_{i=0}^{k}
\sigma_{k,i}\mathcal{Z}_{k+1,i}^2- 2k\sqrt{M}\sigma_{k,k}\mathcal{Z}_{k+1,k}\mathcal{Z}_{k,k} -2\sigma_k\mathcal{Z}_{k+1,k-1}\mathcal{Z}_{k+1,k}\notag\\
&\quad\quad
 -k\sigma_k\sqrt{M}\mathcal{Z}_{k+1,k-1}\mathcal{Z}_{k,k}\notag\\
&\geq \frac14\left(\sigma_k \mathcal{Z}_{k,k}^2 + \sum\limits_{i=0}^{k}
\sigma_{k,i}\mathcal{Z}_{k+1,i}^2\right).
\end{align}

{\bf Step 2.} We present a lower bound for the temporal derivative for $\mathcal{F}_k(t,h_t)$.
\begin{align*}
&\quad -\frac12\frac{\dd}{\dd t}\mathcal{F}_{k}(t,h_t)
= -\frac12\frac{\dd}{\dd t} \left\{\sum\limits_{0\leq i \leq k} \sigma_{k,i}t^{k+2i}||\nabla^{k-i}_v\nabla^i_x h_t||^2
+ 2\sigma_{k}t^{3k-1} \langle \nabla_x^{k-1}\nabla_v h_t, \nabla_x^{k}h_t\rangle\right\}\\
&= -\frac{1}{2t}\left\{\sum\limits_{0\leq i \leq k} (k+2i)\sigma_{k,i}t^{k+2i}||\nabla^{k-i}_v\nabla^i_x h_t||^2
+ 2(3k-1)\sigma_{k}t^{3k-1} \langle \nabla_x^{k-1}\nabla_v h_t, \nabla_x^{k}h_t\rangle\right\} \\
&\quad + \left\{\sum\limits_{0\leq i \leq k} \sigma_{k,i}t^{k+2i}(T_{i,k-i}^A + T_{i,k-i}^B)
+ \sigma_{k}t^{3k-1}(T_{mix}^A +T_{mix}^B)\right\}
\notag\\
&=: -\frac{1}{2t} (\Rom{1}) + (\Rom{2}).
\end{align*}

Now we deal with each term in this expression. First, we denote
\begin{align*}
(\Rom{1}_a) &:= 3k\sigma_{k,k}t^{3k}||\nabla^k_x h_t||^2 = 3k\sigma_{k,k}\mathcal{Z}_{k,k}^2,\\
(\Rom{1}_b) &:= \sum\limits_{0\leq i \leq k-1} (k+2i)\sigma_{k,i}t^{k+2i}||\nabla^{k-i}_v\nabla^i_x h_t||^2
= \sum\limits_{0\leq i \leq k-1} (k+2i)\sigma_{k,i} \mathcal{Z}_{k,i}^2, \\
(\Rom{1}_c) &:=  2(3k-1)\sigma_{k}t^{3k-1} ||\nabla_x^{k-1}\nabla_v h_t||\cdot ||\nabla_x^{k}h_t||
=  2(3k-1)\sigma_{k} \mathcal{Z}_{k,k-1}\mathcal{Z}_{k,k},
\end{align*}
and hence by Cauchy-Schwarz inequality,
\begin{align}\label{EstHe1}
(\Rom{1})
&\leq
(\Rom{1}_a) + (\Rom{1}_b) + (\Rom{1}_c).
\end{align}

Next we decompose the terms in $(\Rom{2})$ by $(\Rom{2})=(\Rom{2}_a) + (\Rom{2}_b)$ with
\begin{align*}
(\Rom{2}_a):=\sum\limits_{0\leq i \leq k} \sigma_{k,i}t^{k+2i}(T_{i,k-i}^A + T_{i,k-i}^B),
\quad
(\Rom{2}_b):=\sigma_{k}t^{3k-1}(T_{mix}^A +T_{mix}^B).
\end{align*}
By the estimates in  Lemma \ref{lemDsptHkE}, we have
\begin{align}
(\Rom{2}_a)&\geq
\sum\limits_{0\leq i \leq k}
\sigma_{k,i}t^{k+2i}\bigg\{\left(||\nabla^i_x \nabla^{k-i+1}_vh||^2 + (k-i)||\nabla^i_x \nabla^{k-i}_vh||^2\right)  \notag\\
&-(k-i) ||\nabla_x^{i+1}\nabla_v^{k-i-1}h ||\cdot ||\nabla_x^{i}\nabla_v^{k-i}h|| \notag\\
& -\sum\limits_{l=1}^{i}\sum\limits_{l_1=0}^{i-l}
\begin{pmatrix}
i-l \\
l_1
\end{pmatrix}
\sqrt{M}\Big(||\nabla_x^{i-l_1-1}\nabla^{k-i+1}_v h|| +
||\nabla_x^{i-l_1}\nabla^{k-i+1}_v h||\Big)
|| \nabla_x^{i}\nabla_v^{k-i}h||
 \bigg\}\notag\\
&\geq \sum\limits_{0\leq i \leq k}
\frac{\sigma_{k,i}}{t}  \bigg\{  t^{k+2i+1}||\nabla^i_x \nabla^{k-i+1}_vh||^2   \notag\\
&-(k-i)\cdot t^{\frac{k}{2}+i+1}||\nabla_x^{i+1}\nabla_v^{k-i-1}h ||\cdot t^{\frac{k}{2}+i}||\nabla_x^{i}\nabla_v^{k-i}h|| \notag\\
& -\sum\limits_{l=1}^{i}\sum\limits_{l_1=0}^{i-l}
\begin{pmatrix}
i-l \\
l_1
\end{pmatrix}
\sqrt{M}\Big(t^{(k+2i-3l_1-2)/2}||\nabla_x^{i-l_1-1}\nabla^{k-i+1}_v h||  \notag\\
&\quad\quad\quad + t^{(k+2i-3l_1+1)/2}||\nabla_x^{i-l_1}\nabla^{k-i+1}_v h||\Big)
\cdot t^{(k+2i)/2}|| \nabla_x^{i}\nabla_v^{k-i}h||
\bigg\}\notag
\end{align}
where the last inequality holds since we assumed that $t\in (0,1]$. In terms of $\mathcal{Z}_{l,i}$, we obtain
\begin{align}
(\Rom{2}_a)&\geq
\sum\limits_{0\leq i \leq k}
\frac{\sigma_{k,i}}{t}  \bigg\{ \mathcal{Z}_{k+1,i}^2- (k-i)\mathcal{Z}_{k,i+1}\mathcal{Z}_{k,i}
\notag\\
&\quad\quad -\sum\limits_{l=1}^{i}\sum\limits_{l_1=0}^{i-l}
\begin{pmatrix}
i-l \\
l_1
\end{pmatrix}
\sqrt{M}\Big(\mathcal{Z}_{k-l_1,i-l_1-1} +
 \mathcal{Z}_{k-l_1+1,i-l_1} \Big)
\mathcal{Z}_{k,i}
\bigg\}
\end{align}
Then we can proceed as in the proof of Lemma \ref{lemDsptHkE} to conclude that
\begin{align}
(\Rom{2}_a)&\geq \frac1t \bigg\{\sum\limits_{i=0}^{k}
\sigma_{k,i}\mathcal{Z}_{k+1,i}^2 -  \sum\limits_{i=0}^{k-1}\sigma_{k,i}\left[k\mathcal{Z}^2 + \mathcal{Z}\mathcal{Z}_{k,k} +2^i\sqrt{M}(2\mathcal{Z}^2 + \mathcal{Z}\mathcal{Z}_{k+1,i}) \right] \notag\\
&\label{EstHe2a}
\quad\quad -2^{k+1}\sqrt{M}\sigma_{k,k}\mathcal{Z}\mathcal{Z}_{k,k} - k\sqrt{M}\sigma_{k,k}\mathcal{Z}_{k+1,k}\mathcal{Z}_{k,k}\bigg\}.
\end{align}
Indeed, such a inequality is a direct consequence of the following observations,
\[
(k-i)\mathcal{Z}_{k,i+1}\mathcal{Z}_{k,i}
\leq\left\{\begin{array}{ll}
0,& \mbox{ if } i=k; \\
\mathcal{Z}\mathcal{Z}_{k,k},& \mbox{ if } i=k-1;\\
k\mathcal{Z}^2,& \mbox{ if } 0\leq i<k-1,
\end{array}
\right.
\]
and
\begin{align*}
\Big(\mathcal{Z}_{k-l_1,i-l_1-1} +
\mathcal{Z}_{k-l_1+1,i-l_1} \Big)
\mathcal{Z}_{k,i}
\leq
\left\{
\begin{array}{ll}
2\mathcal{Z}^2,& \mbox{ if } i<k, l_1\geq 1; \\
(\mathcal{Z}+\mathcal{Z}_{k+1,i})\mathcal{Z}, &\mbox{ if } i<k, l_1=0;\\
2\mathcal{Z}\mathcal{Z}_{k,k},& \mbox{ if } i=k, l_1\geq 1; \\
(\mathcal{Z}+\mathcal{Z}_{k+1,k})\mathcal{Z}_{k,k},& \mbox{ if } i=k, l_1=0.
\end{array}
\right.
\end{align*}
Then the inequality follows from a direct summation, as we did in Lemma \ref{lemDsptHkE}, since
\begin{align*}
\sum\limits_{0\leq i \leq k}
\sigma_{k,i} \big[\mathcal{Z}_{k+1,i}^2- (k-i)\mathcal{Z}_{k,i+1}\mathcal{Z}_{k,i} \big] \geq \sum\limits_{0\leq i \leq k}
\sigma_{k,i} \mathcal{Z}_{k+1,i}^2 - \sum\limits_{0\leq i \leq k-1}
\sigma_{k,i} (k\mathcal{Z}^2+ \mathcal{Z}_{k,k}\mathcal{Z}),
\end{align*}
and
\begin{align*}
&\quad-\sum\limits_{0\leq i \leq k}\sigma_{k,i}\sum\limits_{l=1}^{i}\sum\limits_{l_1=0}^{i-l}
\begin{pmatrix}
i-l \\
l_1
\end{pmatrix}
\sqrt{M}\Big(\mathcal{Z}_{k-l_1,i-l_1-1} +
\mathcal{Z}_{k-l_1+1,i-l_1} \Big)
\mathcal{Z}_{k,i} \\
&\geq -\sum\limits_{0\leq i \leq k-1} \sigma_{k,i} \sum\limits_{l=1}^{i}\sum\limits_{l_1=0}^{i-l}
\begin{pmatrix}
i-l \\
l_1
\end{pmatrix}
\sqrt{M}\Big(2\mathcal{Z} +
\mathcal{Z}_{k+1,i} \Big)
\mathcal{Z} \\
&\quad - \sigma_{k,k}\sum\limits_{l=1}^{k}\sum\limits_{l_1=0}^{k-l}
\begin{pmatrix}
k-l \\
l_1
\end{pmatrix}
\sqrt{M}\cdot 2\mathcal{Z}
\mathcal{Z}_{k,k}  - k\sigma_{k,k}\sqrt{M}\mathcal{Z}_{k+1,k}
\mathcal{Z}_{k,k}\\
&\geq -  \sum\limits_{i=0}^{k-1} 2^i\sigma_{k,i}\sqrt{M} (2\mathcal{Z} + \mathcal{Z}_{k+1,i})\mathcal{Z} -2^{k+1}\sqrt{M}\sigma_{k,k}\mathcal{Z}\mathcal{Z}_{k,k} - k\sqrt{M}\sigma_{k,k}\mathcal{Z}_{k+1,k}\mathcal{Z}_{k,k}.
\end{align*}

Likewise, we apply the estimates in Lemma \ref{lemDsptHkE} to control $(\Rom{2}_b)$.
\begin{align}
&(\Rom{2}_b)=\sigma_{k}t^{3k-1}(T_{mix}^A +T_{mix}^B)\notag\\
&\geq \sigma_{k}t^{3k-1} \bigg\{
-2||\nabla_x^{k-1}\nabla_v^2 h|| \cdot ||\nabla^k_x\nabla_v h|| - ||\nabla_x^{k-1}\nabla_v h|| \cdot ||\nabla^k_xh|| \notag\\
&+||\nabla^k_xh||^2  -\sum\limits_{l=1}^{k-1}\sum\limits_{l_1=0}^{k-l-1}
\begin{pmatrix}
k-1-l \\
l_1
\end{pmatrix}
\sqrt{M}\big(||\nabla_x^{k-l_1-2}\nabla^2_v h||+ ||\nabla_x^{k-l_1-1}\nabla^2_v h||\big)
|| \nabla^k_x h|| \notag \\
& -\sum\limits_{l=1}^{k} \sum\limits_{l_1=0}^{k-l}\begin{pmatrix}
k-l \\
l_1
\end{pmatrix}
||\nabla^{k-1}_x\nabla_v h||\cdot\sqrt{M}
\big(||\nabla_v\nabla_x^{k-l_1-1} h||+ ||\nabla_v\nabla_x^{k-l_1} h||\big)
\bigg\}\notag\\
&\geq \frac{\sigma_{k}}{t} \bigg\{
-2t^{\frac{3k-1}{2}}||\nabla_x^{k-1}\nabla_v^2 h|| \cdot t^{\frac{3k+1}{2}}||\nabla^k_x\nabla_v h|| - t^{\frac{3k-2}{2}}||\nabla_x^{k-1}\nabla_v h|| \cdot t^{\frac{3k}{2}}||\nabla^k_xh||  \notag\\
&\quad +t^{3k}||\nabla^k_xh||^2 \notag\\ &-\sum\limits_{l=1}^{k-1}\sum\limits_{l_1=0}^{k-l-1}
\begin{pmatrix}
k-1-l \\
l_1
\end{pmatrix}
\sqrt{M}\big(t^{\frac{3k-3l_1-4}{2}}||\nabla_x^{k-l_1-2}\nabla^2_v h||+ t^{\frac{3k-3l_1-1}{2}}||\nabla_x^{k-l_1-1}\nabla^2_v h||\big)
t^{\frac{3k}{2}}|| \nabla^k_x h|| \notag \\
& -\sum\limits_{l=1}^{k} \sum\limits_{l_1=0}^{k-l}\begin{pmatrix}
k-l \\
l_1
\end{pmatrix}
t^{\frac{3k-2}{2}}||\nabla^{k-1}_x\nabla_v h||\cdot\sqrt{M}
\big(t^{\frac{3k-3l_1-2}{2}}||\nabla_v\nabla_x^{k-l_1-1} h||+t^{\frac{3k-3l_1+1}{2}} ||\nabla_v\nabla_x^{k-l_1} h||\big)
\bigg\}\notag\\
&= \frac{\sigma_{k}}{t} \bigg\{
-2\mathcal{Z}_{k+1,k-1} \mathcal{Z}_{k+1,k} - \mathcal{Z}_{k,k-1} \mathcal{Z}_{k,k}   +\mathcal{Z}_{k,k}^2 \notag\\
&\quad -\sum\limits_{l=1}^{k-1}\sum\limits_{l_1=0}^{k-l-1}
\begin{pmatrix}
k-1-l \\
l_1
\end{pmatrix}
\sqrt{M}\big(\mathcal{Z}_{k-l_1,k-l_1-2}+ \mathcal{Z}_{k-l_1+1,k-l_1-1}\big)
\mathcal{Z}_{k,k} \notag \\
&\quad -\sum\limits_{l=1}^{k} \sum\limits_{l_1=0}^{k-l}\begin{pmatrix}
k-l \\
l_1
\end{pmatrix}
\sqrt{M}\mathcal{Z}_{k,k-1}
\big(\mathcal{Z}_{k-l_1,k-l_1-1}+\mathcal{Z}_{k-l_1+1,k-l_1}\big)
\bigg\}
\end{align}
As above, we observe that
\begin{align*}
\big(\mathcal{Z}_{k-l_1,k-l_1-2}+ \mathcal{Z}_{k-l_1+1,k-l_1-1}\big)
\mathcal{Z}_{k,k}
\leq
\left\{\begin{array}{ll}
\big(\mathcal{Z}+ \mathcal{Z}_{k+1,k-1}\big)
\mathcal{Z}_{k,k}, &\mbox{ if } l_1=0; \\
 2\mathcal{Z}\mathcal{Z}_{k,k} , &\mbox{ if } l_1>0;
\end{array}
\right.
\end{align*}
\begin{align*}
\mathcal{Z}_{k,k-1}
\big(\mathcal{Z}_{k-l_1,k-l_1-1}+\mathcal{Z}_{k-l_1+1,k-l_1}\big)
\leq
\mathcal{Z}\big(2\mathcal{Z}+ \mathcal{Z}_{k+1,k}\big),
\end{align*}
and then we conclude that
\begin{align} \label{EstHe2b}
(\Rom{2}_b)
&\geq \frac{\sigma_k}{t}\bigg\{-2\mathcal{Z}_{k+1,k-1}\mathcal{Z}_{k+1,k} - \mathcal{Z}\mathcal{Z}_{k,k} + \mathcal{Z}_{k,k}^2-2^{k}\sqrt{M}\mathcal{Z}\mathcal{Z}_{k,k}\notag\\
&\quad  - (k-1)\sqrt{M}\mathcal{Z}_{k+1,k-1}\mathcal{Z}_{k,k}-  2^{k}\sqrt{M}\mathcal{Z} (2\mathcal{Z}+\mathcal{Z}_{k+1,k})\bigg\}.
\end{align}
Combined with the inequalities \eqref{EstHe1} and \eqref{EstHe2a}, we deduce
\begin{align}
-\frac{t}{2}\frac{\dd}{\dd t}\mathcal{F}_{k}(t,h_t)
&
\geq
 -\frac{1}{2} \bigg(3k\sigma_{k,k}\mathcal{Z}_{k,k}^2 + \sum\limits_{i=0}^{k-1} (k+2i)\sigma_{k,i} \mathcal{Z}^2 + 2(3k-1)\sigma_{k} \mathcal{Z}\mathcal{Z}_{k,k}\bigg) \notag\\
&
\quad +  \bigg\{\sum\limits_{i=0}^{k}
\sigma_{k,i}\mathcal{Z}_{k+1,i}^2 -  \sum\limits_{i=0}^{k-1}\sigma_{k,i}\left[k\mathcal{Z}^2 + \mathcal{Z}\mathcal{Z}_{k,k} +2^i\sqrt{M}(2\mathcal{Z}^2 + \mathcal{Z}\mathcal{Z}_{k+1,i}) \right] \notag\\
&
\quad\quad -2^{k+1}\sqrt{M}\sigma_{k,k}\mathcal{Z}\mathcal{Z}_{k,k} - k\sqrt{M}\sigma_{k,k}\mathcal{Z}_{k+1,k}\mathcal{Z}_{k,k}\bigg\} \notag\\
&
\quad + \sigma_k\bigg\{-2\mathcal{Z}_{k+1,k-1}\mathcal{Z}_{k+1,k} - \mathcal{Z}\mathcal{Z}_{k,k} + \mathcal{Z}_{k,k}^2-2^{k}\sqrt{M}\mathcal{Z}\mathcal{Z}_{k,k}\notag\\
&\quad\quad  - (k-1)\sqrt{M}\mathcal{Z}_{k+1,k-1}\mathcal{Z}_{k,k}
-  2^{k}\sqrt{M}\mathcal{Z} (2\mathcal{Z}+\mathcal{Z}_{k+1,k})\bigg\}
\end{align}

{\bf Step 3.} Now we divide the terms in the preceding lower bound into three classes.
\begin{itemize}
	\item The ones  depending only on $\mathcal{Z}_{k,k}$ and $\mathcal{Z}_{k+1,i}$ ($1\leq i\leq k$) are collected in $\mathcal{Q}$, c.f. Step~$1$ \eqref{EstHeQ}. And we have
	\[
	\mathcal{Q}
	\geq \frac14\left(\sigma_k \mathcal{Z}_{k,k}^2 + \sum\limits_{i=0}^{k}
	\sigma_{k,i}\mathcal{Z}_{k+1,i}^2\right)
	\geq \frac{1}{64k^2M}\sigma_k \left(\mathcal{Z}_{k,k}^2 + \sum\limits_{i=0}^{k}
	\mathcal{Z}_{k+1,i}^2\right).
	\]
	\item The ones depending only on $\mathcal{Z}$ are collected below,
	\[
	-\frac{1}{2} \sum\limits_{i=0}^{k-1} (k+2i)\sigma_{k,i} \mathcal{Z}^2
	- \sum\limits_{i=0}^{k-1}\sigma_{k,i}\left[k\mathcal{Z}^2+ 2^{i+1}\sqrt{M}\mathcal{Z}^2 \right] - 2^{k+1}\sigma_k\sqrt{M}\mathcal{Z}^2
	\]
	and it can be further bounded from below by
	\[
	-\left[(2k^2+ 2^{k+1}\sqrt{M})\cdot 64k^2M+ 2^{k+1}\sqrt{M}\right]\sigma_k \mathcal{Z}^2 := -K_0 \sigma_k \mathcal{Z}^2.
	\]
	(Note here $K_0$ depends only on $k$ and $M$.)
	\item The other terms can be bounded from below by
	\[
	-K_1 \sigma_k \mathcal{Z} (\mathcal{Z}_{k,k} + \sum_{i=0}^{k}\mathcal{Z}_{k+1,i})
	\]
	for some positive constant $K_1$ depending only on $k$ and $M$.
\end{itemize}
Putting all together, we have
\[
-\frac{t}{2}\frac{\dd}{\dd t}\mathcal{F}_{k}(t,h_t)
\geq
\sigma_k\bigg\{\frac{1}{64k^2M} \left(\mathcal{Z}_{k,k}^2 + \sum\limits_{i=0}^{k}
\mathcal{Z}_{k+1,i}^2\right)
-K_0 \mathcal{Z}^2
-K_1 \mathcal{Z} \left(\mathcal{Z}_{k,k} + \sum_{i=0}^{k}\mathcal{Z}_{k+1,i}\right)\bigg\}.
\]
Recall by the induction hypothesis, we have
\[
-\frac{\dd}{\dd t}\sum_{l=0}^{k-1}\mathcal{F}_{l}(t,h_t)
\geq \frac{\Lambda_{k-1}}{t} \mathcal{Z}^2.
\]
So it suffices to choose $\sigma_k>0$ small enough to conclude
\[
-\frac{\dd}{\dd t}\sum_{l=0}^{k}\mathcal{F}_{l}(t,h_t)
\geq \frac{\Lambda_{k}}{t} \left(\mathcal{Z}^2 +\mathcal{Z}_{k,k}^2 + \sum\limits_{i=0}^{k}
\mathcal{Z}_{k+1,i}^2\right).
\]
Moreover, it is easy to choose $\sigma_k$ and hence $\Lambda_k$  depending only on $\Lambda_{k-1}, k $ and $M$. For instance, one can choose
\[
\sigma_k = \min\left\{ \frac{\Lambda_{k-1}}{4K_0}, \frac{\Lambda_{k-1}}{128(k+2)k^2MK_1^2} \right\},
\quad\quad \Lambda_k = \min\left\{ \frac{\Lambda_{k-1}}{2}, \frac{\sigma_k}{128k^2M}\right\}.
\]
By the induction hypothesis, $\Lambda_{k-1}$ depends only on $k-1$ and $M$. Therefore the proof is completed.

\end{proof}

\section{An application to the mean field interaction}\label{SectApp1}

In this section we shall focus on the Curie-Weiss model and prove Proposition \ref{PropCW}. Recall the potential $V$ is given by \eqref{CW-V} and \eqref{CW-uw}, i.e.
\begin{equation}
V(x_1,x_2,\cdots,x_N)= \sum_{i} \beta(\frac{x_i^4}{4}-\frac{x_i^2}{2}) - \frac{1}{2N}\sum_{j\neq i}\beta K x_ix_j
\end{equation}
where $x_i\in \RR$ for each $i$, $\beta>0$ is the inverse temperature, and the model is ferromagnetic or antiferromagnetic according to $K > 0$ or $K < 0$.

\subsection{On the Assumption \ref{HypBHess}} This subsection is devoted to prove
\begin{lem}\label{CWLemA2}
	For the potential $V$ defined by \eqref{CW-V} and \eqref{CW-uw}, Assumption 2 holds with
	\[
	M=2020(\beta^{2/3}+ \beta^2 + K^4\beta^2).
	\]
	In particular, $M$ is independent of the number $N$ of particles.
\end{lem}

By direct computation, we have
\begin{align*}
\partial_{x_i} V
&= \beta\Big(x_i^3 - x_i - \frac1N \sum_{j:j\neq i} Kx_j \Big),\\
\partial_{x_i}\partial_{x_j} V
& = \left\{\begin{array}{ll}
		\beta(3x_i^2 -1), & \mbox{ if } i= j;\\
		-\frac{\beta K}{N}, & \mbox{ if } i\neq j,
	\end{array}\right.\\
\partial^3_{x_i}V& = 6\beta x_i,\\
\partial^4_{x_i}V &= 6\beta
\end{align*}
and all the other partial derivatives of $V$ vanish identically. Therefore the inequalities \eqref{IneqVk} in Assumption \ref{HypBHess} can be reduced to the three inequalities below: for any $g\in H^2(\mu)$,
\begin{align}
&\label{CWine1}
\beta^2 \int \sum_i |(3x_i^2-1)\partial_{v_i}g - \frac{K}{N}\sum_{j:j\neq i} \partial_{v_j}g|^2 \dd\mu \leq M \bigg(\int |\nabla_vg|^2 \dd\mu + \int |\nabla^2_{xv}g|^2 \dd\mu\bigg),\\
&\label{CWine2}
36 \beta^2 \int \sum_i x_i^2 |\partial_{v_i}g|^2 \dd\mu \leq M \bigg(\int |\nabla_vg|^2 \dd\mu + \int |\nabla^2_{xv}g|^2 \dd\mu\bigg),\\
& 36 \beta^2\int |\nabla_v g|^2\dd\mu \leq M \bigg(\int |\nabla_vg|^2 \dd\mu + \int |\nabla^2_{xv}g|^2 \dd\mu\bigg). \notag
\end{align}
The third inequality holds trivially for
\begin{equation}\label{CW-V4M}
M\geq 36\beta^2.
\end{equation}
The first and the second one are weighted Poincar\'e inequalities which can be proved by applying the Lyapunov function method. It is well-known that  the constants provided by the Lyapunov function method are usually poor in dimension dependence. However, as shown in our previous work \cite{GLWZ}, it is possible to prove the inequalities in Assumption \ref{HypBHess} with constants independent of the number $N$ of particles. Here we show that this idea works for the Curie-Weiss model as well. 

We need the following weighted Poincar\'e inequalities.

\begin{lem}\label{CWLem24}
There exist constants $M_4',M_4'', M_2', M_2''$ such that
\begin{align}
&\label{CWLemI1}
\int x_i^4 g^2 \dd\mu \leq  M_4'\int g^2 \dd\mu + M_4''\int |\nabla_{x}g|^2 \dd\mu,\\
&\label{CWLemI2}
 \int x_i^2 g^2 \dd\mu \leq  M_2'\int g^2 \dd\mu + M_2''\int |\nabla_{x}g|^2 \dd\mu
\end{align}
for all $g\in H^1(\mu)$. Moreover, the constants $M_4',M_4'', M_2', M_2''$ are explicitly computable and independent of $N$,
\begin{align*}
&M'_4=  2\Big(\frac32+\frac{1}{C}+ 2K^2\Big)^2 +\frac{4}{\beta}, \quad M''_4= \frac{2C}{\beta^2},\\
&M'_2=  \Big(\frac32+\frac{1}{C}+ 2K^2\Big)^2 +\frac{2}{\beta}+ \frac12, \quad M''_2= \frac{C}{\beta^2}+ \frac12
\end{align*}
for any parameter $C>0$.
\end{lem}
\begin{rmq}
	(a). Compared to the results in  \cite{GLWZ}, $\nabla W$ herein is not uniformly bounded and we do not attempt to prove a uniform log-Sobolev inequality for the mean field measure (which would require extra conditions on $\beta$ and $K$). Instead, we get rid of the dependence on $N$ by choosing some suitable coefficients in the application of the Lyapunov function method. Besides, our method also applies to a class of general potentials of mean-field type.\\
	(b). By the method in our proof, a weighted Poincar\'e inequality with the weight $x_i^6$ can be proved with constants independent of the number of particles as well. 
\end{rmq}
\begin{proof}First note that the second inequality follows from the first one since  $x_i^2\leq (x_i^4 + 1)/2$. So it suffices to verify the existence of $M'_4,M_4''$ (being  explicit and being independent of $N$). By density argument, we may assume that $g\in C_c^{\infty}$.
	
\textbf{Claim:}
\textit{Given $C>0$. For all $g\in C_c^{\infty}$, it holds that
\begin{align}\label{CWLemIn}
\int \Big[\beta^2 \Big(x_i^4-x_i^2 -\frac{K}{N}\sum_{j:j\neq i}x_ix_j\Big) - \beta - \frac{\beta^2 x_i^2}{C}\Big]g^2\dd\mu \leq C \int |\partial_{x_i}g|^2\dd\mu .
\end{align}
}

In fact, this is a standard result by the Lyapunov function method. Here we apply a slightly different argument than the one in \cite{GLWZ}. By the identity $\partial_{x_i}V e^{-V}= -\partial_{x_i}(e^{-V})$ and an integration by parts, it holds
\begin{align*}
\int \Phi \partial_{x_i}V g^2 \dd\mu
&= \int \partial_{x_i}(\Phi g^2)\dd\mu = \int \partial_{x_i}\Phi g^2 \dd\mu + 2\int \Phi g\partial_{x_i}g \dd\mu\\
&\leq  \int \partial_{x_i}\Phi g^2 \dd\mu + C\int |\partial_{x_i}g|^2\dd\mu + \frac{1}{C}\int g^2\Phi^2\dd\mu,
\end{align*}
or
\[
\int \bigg(\Phi \partial_{x_i}V- \partial_{x_i}\Phi- \frac{\Phi^2}{C}\bigg)g^2 \dd\mu
\leq C\int |\partial_{x_i}g|^2\dd\mu.
\]
The claim then follows by putting $\Phi= \beta x_i$.

Now we apply the claim to prove the lemma. By replacing $\beta$ by $\frac{\beta}{N}$ and $i$ by $j$ in the inequality \eqref{CWLemIn}, we get for all $j\neq i$,
\begin{align}\label{CWLemIn'}
\int \Big[\frac{\beta^2}{N} \Big(x_j^4-x_j^2 -\frac{K}{N}\sum_{k:k\neq j}x_jx_k\Big) - \frac{\beta}{N} - \frac{\beta^2 x_j^2}{CN^2}\Big]g^2\dd\mu \leq C \int |\partial_{x_j}g|^2\dd\mu.
\end{align}
Combined with inequality \eqref{CWLemIn}, these inequalities imply that
\begin{align*}
C \int |\nabla_{x}g|^2\dd\mu
&\geq\int \Big[\beta^2 \Big(x_i^4-x_i^2 -\frac{K}{N}\sum_{j:j\neq i}x_ix_j\Big) - \beta - \frac{\beta^2 x_i^2}{C}\Big]g^2\dd\mu\\
&\quad\quad +\sum_{j:j\neq i} \int \Big[\frac{\beta^2}{N} \Big(x_j^4-x_j^2 -\frac{K}{N}\sum_{k:k\neq j}x_jx_k\Big) - \frac{\beta}{N} - \frac{\beta^2 x_j^2}{CN^2}\Big]g^2\dd\mu\\
&:= \int \beta \Psi g^2\dd\mu
\end{align*}
where the expression $\Psi$ is given by
\[
\Psi=\beta \Big(x_i^4-x_i^2 -\frac{K}{N}\sum_{j:j\neq i}x_ix_j\Big) -1- \frac{\beta x_i^2}{C}
+\sum_{j:j\neq i}  \Big[\frac{\beta}{N} \Big(x_j^4-x_j^2 -\frac{K}{N}\sum_{k:k\neq j}x_jx_k\Big) - \frac{1}{N} - \frac{\beta x_j^2}{CN^2}\Big].
\]
The positive terms involving $x_j^4$ helps to control the crossed term involving $x_ix_j$, while the prefactor $\frac1N$ introduced in \eqref{CWLemIn'}  is crucial in the above expression to obtain dimensionless constants in the lemma. 

Then the first inequality in the lemma follows as soon as we prove
\[
\Psi\geq \frac{\beta}{2}x_i^4 - \beta \Big(\frac32+\frac{1}{C}+ 2K^2\Big)^2 -2.
\]
Indeed, using $2Kx_ix_j\leq K^2x_j^2 + x_i^2$ and $2Kx_jx_k\leq K^2x_k^2 + x_j^2$, we have
\begin{align*}
\Psi +2
&\geq
\beta \Big(x_i^4-\frac32 x_i^2 -\frac{K^2}{N}\sum_{j:j\neq i}x_j^2\Big) - \frac{\beta x_i^2}{C}
+\sum_{j:j\neq i}  \Big[\frac{\beta}{N} \Big(x_j^4-\frac32 x_j^2 -\frac{K^2}{N}\sum_{k:k\neq j}x_k^2\Big) - \frac{\beta x_j^2}{CN^2}\Big] \\
&= \beta \Big(x_i^4-\frac32 x_i^2\Big) - \frac{\beta x_i^2}{C} - \frac{(N-1)K^2\beta}{N^2} x_i^2\\
&\quad -\frac{\beta K^2}{N}\sum_{j:j\neq i}x_j^2 + \frac{\beta}{N} \sum_{j:j\neq i}  \Big(x_j^4-\frac32 x_j^2\Big) - \frac{\beta K^2}{N^2}\sum_{j:j\neq i} \sum_{k:k\neq j, k\neq i}x_k^2 - \sum_{j:j\neq i} \frac{\beta x_j^2}{CN^2}\\
&\geq \beta \Big[x_i^4-\big(\frac32+\frac{1}{C}+ K^2\big)x_i^2\Big]
+ \frac{\beta}{N}\sum_{j:j\neq i}\bigg[   x_j^4- \Big(K^2 +\frac{3}{2} + \frac{K^2(N-2)}{N} + \frac{1}{CN} \Big) x_j^2 \bigg]\\
&\geq \frac{\beta}{2}\Big[ x_i^4 - \big(\frac32+\frac{1}{C}+ K^2\big)^2\Big]
- \frac{\beta}{N}\sum_{j:j\neq i}\frac14\Big(K^2 +\frac{3}{2} + \frac{K^2(N-2)}{N} + \frac{1}{CN}\Big)^2\\
&\geq  \frac{\beta}{2}\Big[ x_i^4 - \big(\frac32+\frac{1}{C}+ K^2\big)^2\Big]
- \frac{\beta}{4}\Big(2K^2 +\frac{3}{2}  + \frac{1}{C}\Big)^2\\
&\geq \frac{\beta}{2}x_i^4 -\beta \big(\frac32+\frac{1}{C}+ 2K^2\big)^2.
\end{align*}
Thus we have finished the proof of Lemma \ref{CWLem24}.
\end{proof}

With Lemma \ref{CWLem24} at hand, it is a routine to prove the inequalities \eqref{CWine1} and \eqref{CWine2}. Indeed, the inequality \eqref{CWine2} follows directly  from \eqref{CWLemI2} with
\begin{equation}\label{CW-V3M}
M\geq \max\Big\{18\beta^2 \Big(M_4'+1\Big), 18\beta^2 \Big(M_4''+1\Big)\Big\}.
\end{equation}
For the inequality \eqref{CWine1}, Lemma \ref{CWLem24} yields that
\begin{align*}
\sum_i\int  x_i^4|\partial_{v_i}g|^2 \dd\mu
&\leq  M_4'\sum_i\int |\partial_{v_i}g|^2 \dd\mu + M_4''\sum_i\int |\nabla_{x}\partial_{v_i}g|^2\\
&= M_4'\int |\nabla_vg|^2 \dd\mu + M_4''\int |\nabla_{x}\nabla_vg|^2 \dd\mu
\end{align*}
and hence
\begin{align*}
&\quad \int |\nabla^2V\cdot\nabla_vg|^2\dd\mu=\beta^2 \int \sum_i \Big|(3x_i^2-1)\partial_{v_i}g - \frac{K}{N}\sum_{j:j\neq i} \partial_{v_j}g\Big|^2 \dd\mu \\
&\leq
2\beta^2 \sum_i\int  \Big[(3x_i^2-1)^2|\partial_{v_i}g|^2
+   \Big|\frac{K}{N}\sum_{j:j\neq i} \partial_{v_j}g\Big|^2 \Big] \dd\mu \\
&\leq 18\beta^2 \int \sum_i x_i^4|\partial_{v_i}g|^2 \dd\mu
+   2(1+K^2)\beta^2\int|\nabla_{v}g|^2 \dd\mu\\
&\leq \beta^2(18M_4'+2+2K^2)\int |\nabla_vg|^2 \dd\mu
+   18\beta^2M_4''\int|\nabla^2_{xv}g|^2 \dd\mu.
\end{align*}
That is, the inequality \eqref{CWine1} holds with
\begin{equation}\label{CW-V2M}
	M\geq \max\Big\{2\beta^2(9M_4'+1+K^2), 18\beta^2M_4''\Big\}
\end{equation}
where $M_4'$ and $M_4''$, given in Lemma \ref{CWLem24}, are independent of the number $N$ of particles.

As a conclusion, Assumption \ref{HypBHess} holds with any $M$ satisfying \eqref{CW-V2M}, \eqref{CW-V3M} and \eqref{CW-V4M}, namely,
\begin{align*}
M&\geq \max\{ 18\beta^2M_4' + 2\beta^2 + 2\beta^2K^2, 18\beta^2M_4'', 18\beta^2M_4' + 18\beta^2, 18\beta^2M_4''+ 18\beta^2, 36\beta^2\}.
\end{align*}
This holds for instance if we take  $C=\beta^{2/3}$ and
\[
M= 2020(\beta^{2/3}+ \beta^2 + K^4\beta^2).
\]
Therefore the proof of Lemma \ref{CWLemA2} is completed.

\subsection{Proof of Proposition \ref{PropCW}}
To verify the Assumption \ref{HypPI}, we quote the results in \cite[Section 2.2, Example 1]{GLWZ1} concerning Poincar\'e inequality for the invariant measure $\mu$. The proof is omitted. (Note that one should replace the $K$ therein by $\frac{N-1}{N}K$.)

\begin{lem}[c.f. \cite{GLWZ1}]\label{CWLemA1}
	For the potential $V$ defined by \eqref{CW-V} and $\eqref{CW-uw}$, the probability measure $\dd \nu(x)= \frac{1}{Z_1} e^{-V(x)}\dd x$ satisfies a Poincar\'e inequality with some constant $\kappa$
	\[
	\kappa \leq \left\{
	\begin{array}{ll}
	\left(\frac{\sqrt{\pi}}{\sqrt{\beta}} e^{\beta/4}+\frac{\beta K}{N}\right)^{-1}, \ &\text{ if } K<0,\\
	\left(\frac{\sqrt{\pi}}{\sqrt{\beta}} e^{\beta/4} -\beta \frac{N-1}{N}K\right)^{-1},\ &\text{ if }  K>0,
	\end{array}
	\right.
	\]
	provided that the RHS expressions are strictly positive.
\end{lem}
\begin{proof}[Proof of Proposition \ref{PropCW}]
	By Lemma \ref{CWLemA1}, we know
	\begin{itemize}
		\item in the case $K < 0$, when $N$ is sufficiently large, $\kappa$ in Assumption \ref{HypPI} can be chosen as
		\[
		\kappa = \frac12 \frac{\sqrt{\beta}}{\sqrt{\pi}} e^{-\beta/4};
		\]
		\item in the case $K > 0$, suppose $\beta$ is small in the sense that
		\[
		\lambda_1:=\frac{\sqrt{\pi}}{\sqrt{\beta}} e^{\beta/4} -\beta K>0.
		\]
		(Note that the expression above tends to infinity as $\beta\rightarrow 0$.) Then Assumption $\ref{HypPI}$ holds with $\kappa=1/\lambda_1$.
	\end{itemize}
   Note also Assumption \ref{HypBHess} holds with $M$ given by Lemma \ref{CWLemA2} where $M$ does not depend on $N$. Therefore, in both cases, Theorems \ref{thmHigh} and Theorem \ref{thmReg} apply, and the constants in the estimates are independent of the number $N$ of particles.
\end{proof}


\section{Appendix A: On short time asymptotic of the fundamental solutions}
\label{SectFSol}
In this appendix we show that the exponents in the short-time regularity estimates \eqref{EqThm} of H\'erau-Villani type are optimal, with the help of the fundamental solutions of the kinetic Fokker-Planck equation in the case of quadratic potentials.

\subsection{The fundamental solution of Kolmogorov}
Recall that the fundamental solution to Kolmogorov's operator $\partial_t + v\cdot \nabla_x +\eta\cdot\nabla_v -\theta\Delta_v$ with constant vector $\eta$ and constant $\theta>0$, starting from the initial measure $\delta_{(x_0,v_0)}$, is given by
\begin{align}
&\left(\frac{\sqrt{3}}{2\pi\theta t^2}\right)^{d}
\exp\Bigg\{-\frac{1}{\theta}\bigg[\frac{3|x-x_0-v_0t - \eta t^2/2|^2}{t^3} \notag\\
& \quad\quad- \frac{3(x-x_0-v_0t - \eta t^2/2)\cdot (v-v_0- \eta t)}{t^2}
 + \frac{|v-v_0- \eta t|^2}{t}\bigg]\Bigg\}.
\end{align}
(In some sense, this formula in the simplest case of $\theta=1,x_0=v_0=\eta=0$ is all we need to test optimality in short time.) It corresponds to the diffusion process $(X_t,V_t)_{t\geq 0}$ evolving according to the SDE
\begin{align*}
\left\{
\begin{array}{l}
\dd X_t= V_t\dd t\\
\dd V_t= \eta \dd t + \sqrt{2\theta}\dd B_t
\end{array}
\right.
\end{align*}
subject to the initial condition $(X_0,V_0)=(x_0,v_0)$, where $(B_t)_{t\geq 0}$ is a standard Brownian motion. The mean and covariance matrix of $(X_t,Y_t)$ are given by
\[
\mbox{mean} (\begin{pmatrix}
	X_t\\
	V_t
\end{pmatrix})= (x_0 + v_0t+\frac{\eta t^2}{2}, v_0+\eta t)^{\mathsf{T}},
\quad
\mbox{Cov}(\begin{pmatrix}
X_t\\
V_t
\end{pmatrix})
=\begin{pmatrix}
\frac23 \theta t^3 I_{d} & \theta t^2 I_{d}\\
\theta t^2 I_{d} & 2\theta t I_{d}
\end{pmatrix}
\]
where $I_{d}$ is the identity matrix in $\RR^d$. Note that the fundamental solution can be written as
\begin{align}
&\left(\frac{\sqrt{3}}{2\pi \theta t^2}\right)^{d}
\exp\Bigg\{-\frac{1}{\theta}\bigg[\frac{3|x-x_0-(v+v_0)t/2|^2}{t^3}
+ \frac{|v-v_0- \eta t|^2}{4t}\bigg]\Bigg\}
\end{align}
which is the form given by A.N.~Kolmogorov \cite{Kol34} in 1934.

\subsection{Sharpness of Theorem \ref{thmReg}}
Let us show that the exponents in the regularity estimates \eqref{EqThm} are optimal in short time. Set the potential $V(x)=\frac{\omega_0^2}{2} |x|^2$ (with $\omega_0>0$). Then it is well-known that the fundamental solution can be written explicitly for the evolution equation \eqref{Eq-kFP}. Indeed it is standard to derive the explicit formula by Ito's stochastic calculus: the fundamental solution starting from the initial measure $\delta_{(x_0,v_0)}$ is given by the law of the $\RR^d\times\RR^d$-valued diffusion process $(X_t,V_t)_{t\geq 0}$ evolving according to
\begin{align*}
\dd
\begin{pmatrix}
X_t\\
V_t
\end{pmatrix}
=
\begin{pmatrix}
0 &1\\
-\omega_0^2 & -1
\end{pmatrix}
\begin{pmatrix}
X_t\\
V_t
\end{pmatrix}\dd t + \sqrt{2}\dd \begin{pmatrix}
0\\
B_t
\end{pmatrix}
\end{align*}
with initial datum $(X_0,V_0)=(x_0,v_0)$, where $(B_t)_{t\geq 0}$ is a standard Brownian motion on $\RR^d$. The above SDE can be solved explicitly by
\[
\begin{pmatrix}
X_t\\
V_t
\end{pmatrix}
= e^{\Xi t}
\begin{pmatrix}
x_0\\
v_0
\end{pmatrix}
+
\sqrt{2}\int_0^t e^{\Xi (t-s)}
\begin{pmatrix}
0\\
\dd B_s
\end{pmatrix}
\]
where we denote $\Xi=\begin{pmatrix}
0 &1\\
-\omega_0^2 & -1
\end{pmatrix}$. Moreover, the law of $(X_t,V_t)$ is determined by its mean value and covariance matrix. Instead of reproducing the exact (but complex) form of the fundamental solution, we present its asymptotic form from which it is easier to glean information about the asymptotic behaviour in short time. To avoid heavy expressions we only consider the case $x_0=v_0=0$. We assume furthermore that $\omega_0\neq \frac12$ (the case $\omega_0=\frac12$ is a little bit different). Then the matrix $\Xi$ has two distinct non-zero eigenvalues $\lambda_1$ and $\lambda_2$  with respective eigenvector $\xi_1$ and $\xi_2$, i.e.
\begin{align*}
\begin{pmatrix}
0 &1\\
-\omega_0^2 & -1
\end{pmatrix}
(\xi_1,\xi_2)=
(\xi_1,\xi_2)
\begin{pmatrix}
\lambda_1 & 0\\
0& \lambda_2
\end{pmatrix}.
\end{align*}
Let $\alpha_1, \alpha_2$ be two complex numbers such that
\[
\alpha_1\xi_1 + \alpha_2 \xi_2
=\begin{pmatrix}
0\\
1
\end{pmatrix}
\]
(for instance, one may set $\xi_i= (1, \lambda_i)^{\mathsf{T}}$, $\alpha_1= -\alpha_2= 1/(\lambda_1-\lambda_2)$) and so
\begin{align}
	\begin{pmatrix}
		\sum_k \lambda^n_k\alpha_k\xi_{k1}\\
		\sum_k \lambda^n_k\alpha_k\xi_{k2}
	\end{pmatrix}
	&=
	\begin{pmatrix}
		\xi_{11} & \xi_{21}  \\
		\xi_{12} &\xi_{22}
	\end{pmatrix}
	\begin{pmatrix}
		\lambda_1&0\\
		0&\lambda_2
	\end{pmatrix}^n
	\begin{pmatrix}
		\alpha_1\\
		\alpha_2
	\end{pmatrix}\notag\\
	&\label{FSolCoef1}
	= \begin{pmatrix}
		0 &1\\
		-\omega_0^2 & -1
	\end{pmatrix}^n
	(\xi_1,\xi_2)
	\begin{pmatrix}
		\alpha_1\\
		\alpha_2
	\end{pmatrix}
	= \begin{pmatrix}
		0 &1\\
		-\omega_0^2 & -1
	\end{pmatrix}^n
	\begin{pmatrix}
		0\\
		1
	\end{pmatrix}.
\end{align}

Now we rewrite the martingale part of the diffusion process as
\[
\sqrt{2}\int_0^t e^{\Xi (t-s)}
\begin{pmatrix}
	0\\
	\dd B_s
\end{pmatrix}
=
\sqrt{2}\int_0^t \left(\alpha_1\xi_1 e^{\lambda_1(t-s)} + \alpha_2 \xi_2 e^{\lambda_2(t-s)}\right)\dd B_s,
\]
and we compute the covariance matrix by Ito's isometry. Set $\xi_{k}= (\xi_{k1}, \xi_{k2})^{\mathsf{T}}$, $k=1,2$. Then the covariance matrix is in a $2$-by-$2$ block form with the $(i,j)$-element $\gamma_{ij}(t) I_{d}$ where $1\leq i,j\leq 2$ and
\begin{align}
\gamma_{ij}(t)
&= 2\int_0^t \left(\alpha_1\xi_{1i} e^{\lambda_1(t-s)} + \alpha_2 \xi_{2i} e^{\lambda_2(t-s)}\right)
\left(\alpha_1\xi_{1j} e^{\lambda_1(t-s)} + \alpha_2 \xi_{2j} e^{\lambda_2(t-s)}\right)\dd s \notag\\
&\label{FSolCovMat}
=2 \sum_{k,l} \frac{1}{\lambda_k +\lambda_l} (e^{(\lambda_k+\lambda_l)t}-1) \alpha_k\alpha_l \xi_{ki}\xi_{lj}\\
&= 2 \sum_{k,l} \left[t + \frac{1}{2}(\lambda_k+\lambda_l)t^2 + \frac{1}{6} (\lambda_k+\lambda_l)^2t^3 + o(t^3)\right] \alpha_k\alpha_l \xi_{ki}\xi_{lj}\notag\\
&= 2a^{(i,j)}_1t + a^{(i,j)}_2t^2 + \frac{1}{3}a^{(i,j)}_3t^3 + o(t^3)\notag
\end{align}
as $t\rightarrow 0+$, where the second last line follows by a Taylor expansion and the constants $a^{(i,j)}_1, a^{(i,j)}_2, a^{(i,j)}_3$ are given by
\begin{align*}
a^{(i,j)}_1&
=\sum_{k}\alpha_k \xi_{ki} \cdot  \sum_{l}\alpha_l \xi_{lj}, \\
a^{(i,j)}_2&
=\sum_{k}\alpha_k \lambda_k\xi_{ki}\cdot \sum_{l}\alpha_l \xi_{lj}
+ \sum_{k}\alpha_k \xi_{ki}\cdot  \sum_{l}\alpha_l \lambda_l\xi_{lj},\\
a^{(i,j)}_3&=
 \sum_{k}\alpha_k \lambda_k^2\xi_{ki}\cdot \sum_{l}\alpha_l \xi_{lj}
+2 \sum_{k}\alpha_k \lambda_k\xi_{ki}\cdot  \sum_{l}\alpha_l \lambda_l\xi_{lj}
+  \sum_{k}\alpha_k \xi_{ki}\cdot  \sum_{l}\alpha_l \lambda_l^2\xi_{lj}.
\end{align*}
By \eqref{FSolCoef1} (with $n=0,1$), we have
\begin{align*}
a^{(1,1)}_1= 0,\quad  a^{(1,2)}_1=0,\quad  a^{(2,2)}_1=1;\quad a^{(1,1)}_2= 0, \quad a^{(1,2)}_2=1;\quad a^{(1,1)}_3=2.
\end{align*}
It follows that, omitting the identity $I_d$, the covariance matrix has the asymptotic form
\begin{align}\label{FSolCov0}
\Gamma
:=\begin{pmatrix}
\gamma_{11}(t) &  \gamma_{12}(t)\\
\gamma_{21}(t) & \gamma_{22}(t)
\end{pmatrix}
=\begin{pmatrix}
\frac23 t^3 + o(t^3) &  t^2 + o(t^2)\\
  t^2 + o(t^2) & 2t + o(t)
\end{pmatrix}
\end{align}
as $t\rightarrow 0+$. (We remark that this matrix has the same asymptotic behaviour at the starting time as of the covariance matrix  with $\theta=1$ in Kolmogorov's fundamental solutions, just as expected.) We then calculate the inverse of $\Gamma$
and find the law of $(X_t,V_t)$ with $(X_0,V_0)=(0,0)$ is given by the following asymptotic form
\begin{align}\label{FSolApp}
G_0(x,v,t):=&\left(\frac{\sqrt{3}}{2\pi}\right)^{d} (t^{-2d} + o(t^{-2d}))
\exp\Bigg\{-\bigg[\frac{3|x|^2}{t^3 + o(t^3)}
- \frac{3x\cdot v}{t^2 + o(t^2)}
+ \frac{|v|^2}{t + o(t)}\bigg]\Bigg\}.
\end{align}
Note that those $o(t^{k})$'s above are indeed (different) functions depending only on $t$, in particular, independent of $x$ and $v$. 
Moreover, the density $G(t,x,v)$ with respect to the invariant measure shares the very same asymptotic form in \eqref{FSolApp} (with different $o(t^k)$'s). One may reformulate $G(t,x,v)$ as
\begin{align*}
&G(t,x,v)= t^{-2d} H(t,x',v')\\
&H(t,x',v'):=\left(\frac{\sqrt{3}}{2\pi}\right)^{d} (1 + o(1))
\exp\Bigg\{-\bigg[\frac{3|x'|^2}{1 + o(1)}
- \frac{3x'\cdot v'}{1 + o(1)}
+ \frac{|v'|^2}{1 + o(1)}\bigg]\Bigg\}
\end{align*}
as $t\rightarrow 0+$, where the $o(1)$'s are again functions of $t$, and the variables are defined by
\[
x':= t^{-\frac32} x, \quad v':=t^{-\frac12} v.
\]
So $\nabla^i_x\nabla^j_v G= t^{-2d-(3i+j)/2} \nabla^i_{x'}\nabla^j_{v'}H$, and $\dd x'\dd v'= t^{-2d}\dd x\dd v$.
We thus conclude
\begin{align}
t^{2d+3i+j} \int |\nabla^i_x\nabla^j_vG(t,x,v)|^2 \dd\mu(x,v)
&=
Z'\int |\nabla^i_{x'}\nabla^j_{v'} H(t,x',v')|^2 e^{-\frac{\omega_0^2 t^3}{2}|x'|^2 - \frac{t}{2}|v'|^2} \dd x' \dd v'\notag\\
&\label{FSolLim0}
\longrightarrow Z'\int |\nabla^i_{x'}\nabla^j_{v'} H(0,x',v')|^2  \dd x' \dd v'
\end{align}
with $Z'= \omega_0^{d} (2\pi)^{-d}$, as $t\rightarrow 0+$, where the convergence is guaranteed by Lebesgue dominated convergence theorem with the help of the structure of $\nabla^i_{x'}\nabla^j_{v'} H(x',v',t)$. A simple observation is that  the above limit is non-zero. 
\begin{rmq}One may find that this convergence relation \eqref{FSolLim0} holds for general initial datum $\delta_{(x_0,v_0)}$ as well, by a translation in the variables. It is also clear that  the invariant measure $\mu$ therein can be replaced by the Lebesgue measure, while $G$ replaced by $G_0$. By a similar argument of the analysis below, the exponent $(2d+3i+j)/2$ might be seen as a possible optimal exponent in the regularity estimates with $L^1$ initial data, which is consistent with the conjectured exponents in \cite[Section A.21, Remark A.16]{Villani}.
\end{rmq}

From \eqref{FSolLim0}, it is easy to conclude that the exponents in the regularity estimates \eqref{EqThm} in Theorem \ref{thmReg} are sharp in short time. For a detailed justification, one may argue as follows. Suppose there exist a constant $C$ and a function $\varphi(t)\in C((0,t_1], \RR_+)$ such that for any solution $h_t$ to \eqref{Eq-kFP*} with initial datum $h_0\in L^2(\mu)$,
\begin{equation}
||\nabla_x^{i}\nabla_v^{j}h_t||^2_{L^2(\mu)}\leq C\varphi(t) ||h_0||^2_{L^2(\mu)}, \quad \mbox{ for } 0<t\leq t_1.
\end{equation}
Then, by \eqref{FSolLim0}, there exists sufficiently small $t_2>0$ such that for $0<s\leq t_2$
\begin{align*}
s^{2d+3i+j}||\nabla_x^{i}\nabla_v^{j}G(s,x,v)||^2_{L^2(\mu)}
&\geq
\frac12 Z'||\nabla_{x'}^{i}\nabla_{v'}^{j}H(0,x,v)||^2_{L^2(\dd x'\dd v')},\\
s^{2d}||G(s,x,v)||^2_{L^2(\mu)}
&\leq 2Z'||H(0,x',v')||^2_{L^2(\dd x'\dd v')}.
\end{align*}
Now consider $h_t=G(\varepsilon+t,x,v)$ with initial datum $h_0= G(\varepsilon,x,v)\in L^2(\mu)$. By the assumption above, it holds for $0<\varepsilon,t\leq  \min\{t_1,t_2\}/2$ that
\begin{align*}
C\varphi(t)
\geq\frac{||\nabla_x^{i}\nabla_v^{j}h_t||^2_{L^2(\mu)}}{||h_0||^2_{L^2(\mu)}} &= \frac{\varepsilon^{2d}}{(t+\varepsilon)^{2d+3i+j}} \cdot
\frac{(t+\varepsilon)^{2d+3i+j}||\nabla_x^{i}\nabla_v^{j}G(\varepsilon+t,x,v)||^2_{L^2(\mu)}}{\varepsilon^{2d}||G(\varepsilon,x,v)||^2_{L^2(\mu)}}\\
&\geq \frac{\varepsilon^{2d}}{(t+\varepsilon)^{2d+3i+j}}\cdot \frac14
\frac{||\nabla_{x'}^{i}\nabla_{v'}^{j}H(0,x,v)||^2_{L^2(\dd x'\dd v')}}{||H(0,x',v')||^2_{L^2(\dd x'\dd v')}}.
\end{align*}
We then take $\varepsilon=t$ to deduce that
\begin{equation}
\varphi(t)\geq ct^{-(3i+j)}, \quad \forall t\in (0,\min\{t_1,t_2\}/2]
\end{equation}
for some constant $c>0$. That is, the exponent $(3i+j)/2$ is optimal in the short time regularity estimates.

\section{Appendix B: A technical lemma}
\label{SectTechLem}
The following observation might be helpful in several situations: \emph{a real-valued lower-triangular matrix $S= (s_{ij})_{1\leq i,j\leq N}$, with positive diagonal elements, is positive in the sense of quadratic forms whenever there exists constants  $ \{k_{ij}\geq0 \,|\, 1\leq i,j\leq N, i\neq j\}$ such that $\sum_{j:j\neq i} k_{ij}\leq 1$ and
	\[
	|s_{ij}|^2 \leq 4 k_{ij}s_{ii} k_{ji}s_{jj}, \mbox{ for all } i>j.
	\]}
Recall we may assume that $M\geq 1$.

\begin{lem}\label{LemApp1}For $b>0$, $b= \frac{a}{64M}$ and $c=\frac{a}{1024M^2}$, we have
	\[
	S_w:=\begin{pmatrix}
	b         & 0      &0\\
	-b\sqrt{M} & a   &0\\
	-2c\sqrt{M}    & -2b  &c
	\end{pmatrix}
	\geq \text{Diag}(\frac12b, \frac14 a, \frac14 c)
	\]
	in the sense of quadratic forms.
\end{lem}

\begin{proof}
It is equivalent to the positiveness of the $3$-by-$3$ matrix
\[
\begin{pmatrix}
\frac12 b           & 0            &0\\
-b\sqrt{M}      & \frac34 a           &0\\
-2c\sqrt{M}     & -2b   &\frac34 c
\end{pmatrix}
\]
in the sense of quadratic forms. In this case we may take $k_{12}=k_{13}=\frac12$, $k_{21}=\frac13$, $k_{23}=\frac23$, $k_{31}=\frac13$, $k_{32}=\frac23$ and it is easy to verify the set of inequalities
\begin{align*}
	&  b^2M \leq 4\times k_{12} \cdot\frac12 b \times k_{21}\cdot \frac34 a = \frac14 ba , \\
	&  4c^2M \leq 4 \times k_{13}\cdot \frac12 b \times k_{31}\cdot \frac34 c=  \frac14 bc,\\
	& 4b^2 \leq 4\times k_{23} \cdot \frac34 a  \times k_{32} \cdot\frac34 c=ac.
\end{align*}
In fact, these inequalities are equivalent to
\[
4M b\leq a,\quad 16M c\leq b, \quad 4 b^2\leq ac
\]
which hold since $a=64Mb$, $b=16M c$.
\end{proof}

\textbf{Acknowledgement.} The author would like to thank his PhD supervisors, Prof.~Arnaud Guillin and Prof.~Liming Wu, for their support and encouragement of this work. Especially, he would like to express his appreciation for Prof.~Wu's great help in the redaction. He also would like to thank Prof.~Cl\'ement Mouhot for his valuable comments.

\bibliographystyle{plain}

\end{document}